\documentclass[pdflatex,12pt,a4paper,twoside]{article}
\usepackage{authblk}
\usepackage{amsmath,amsfonts,amssymb,amsthm}
\usepackage{ifthen}
\usepackage{graphicx}
\usepackage{epsfig}
\usepackage{nicefrac}
\usepackage{mathrsfs}
\usepackage[a4paper,left=30mm,right=30mm,top=26mm,bottom=26mm]{geometry}
\usepackage{amsmath}
\usepackage{amsfonts}
\newcommand{\R}{\mathbb{R}}
\newcommand{\N}{\mathbb{N}}
\newcommand{\C}{\mathbb{C}}
\newcommand{\Z}{\mathbb{Z}}
\newcommand{\eps}{\varepsilon}

\newcommand{\supp}{\mathrm{supp}}

\def\XXint#1#2#3{{\setbox0=\hbox{$#1{#2#3}{\int}$}
     \vcenter{\hbox{$#2#3$}}\kern-.5\wd0}}

\usepackage[
    bookmarks,
    colorlinks=true,
    bookmarksopen=true,
    bookmarksopenlevel=1,
    linkcolor=black,
    citecolor=blue,
    urlcolor=blue,
    filecolor=blue,
    anchorcolor=red,
    pdfstartview=FitH,
    pdfpagelayout=OneColumn,
    plainpages=false, 
    pdfpagelabels, 
    hypertexnames=false, 
    linktocpage, 
]{hyperref}

\newtheorem{theorem}{Theorem}[section]
\newtheorem{definition}[theorem]{Definition}
\newtheorem{lemma}[theorem]{Lemma}
\newtheorem{proposition}[theorem]{Proposition}
\newtheorem{corollary}[theorem]{Corollary}
\newtheorem{assumption}[theorem]{Assumption}

\newtheorem{remark}[theorem]{Remark}
\newtheorem{example}{Example}

\numberwithin{equation}{section}

\title{High-order homogenization in   optimal control  by  the Bloch wave method}

\author{Agnes Lamacz-Keymling and Irwin Yousept\thanks{Universit\"at Duisburg-Essen, Fakult\"at f\"ur Mathematik, Thea-Leymann-Stra\ss e 9, D-45127 Essen, Germany.}}

\begin{document}
\maketitle

\begin{abstract}
This article examines a linear-quadratic elliptic optimal control problem in which the cost functional and the state equation involve a highly oscillatory periodic coefficient $A^\eps$.  
The small parameter $\eps>0$ denotes the periodicity length. We propose a high-order effective control problem with constant coefficients that provides an approximation of the original one  
with error $O(\eps^M)$, where $M\in\N$ is as large as one likes. 
Our analysis relies on a Bloch wave expansion of the optimal solution and is performed in two steps. 
In the first step, we expand the lowest Bloch eigenvalue in a Taylor series to obtain a high-order effective optimal control problem. 
In the second step, the original and the effective problem are rewritten in terms of the Bloch and the Fourier transform, respectively. 
This allows for a direct comparison of the optimal control problems via the corresponding variational inequalities. 
\end{abstract}

\smallskip {\bf Keywords:} Optimal control, periodic homogenization, Bloch
analysis

  \smallskip
  {\bf MSC: 35B27,35P05, 49J20} 

\pagestyle{myheadings} 
\thispagestyle{plain} 

\markboth{A.\,Lamacz, I.\,Yousept}{High-order homogenization in optimal control  by the Bloch wave method}

\section{Introduction}

Many modern key technologies   call for mathematical modeling through partial differential equations (PDEs) with macro- and microstructures.    The simplest way to decode a microstructure is to consider periodic coefficients 
with a small periodicity length $\eps>0$. 
In such a situation, the central question concerns the effective or asymptotic behavior  of the problem in the homogenization limit $\eps\to 0$.  The periodic homogenization theory was mainly developed in the late '70s focusing on the effective behavior of elliptic PDEs  (see \cite{BLP,BP, SP}). Since then, the theory has been 
generalized to various types of equations and more complex models, including perforated domains \cite{CDGZ}, high contrast 
media, and singular geometries \cite{BS,LS}, just to mention a few. 

The homogenization theory plays as well a profound role in  the optimal control  of  PDEs with highly
oscillatory coefficients. Such problems are particularly encountered in electrical and electronic engineering applications \cite{Y13,Y17}, which require a substantial extension of the developed techniques by the homogenization theory. 
We refer to \cite{BDFS, KM,KR,KP1,KL} for the interplay of homogenization and optimal control based on weak convergence methods. However, those results are only able to characterize the behavior in the limit $\eps\to 0$. To the best of the authors' knowledge, neither
convergence rates nor higher-order approximations are available in the literature. In the classical homogenization theory, higher-order approximations are achieved by 
adding highly oscillatory terms to the effective solution, the so-called correctors, which solve microscopic cell problems and   characterize the oscillatory behavior 
of the original solution. Unfortunately, this strategy does not apply to   optimal control problems, since the control functions are 
typically restricted to an admissible set, and adding correctors to the effective control function will destroy its admissibility. 
To circumvent this difficulty, we propose the use of a spectral method involving the so-called Bloch waves, which allows for a variety of 
$L^2$-conditions in the admissible set. Examples are provided in Section \ref{sec:assumptions} below.  

This paper is aimed at  the high-order asymptotic behavior of a linear-quadratic elliptic optimal control problem with $\eps$-periodic coefficients and $\eps$-dependent admissible sets.
More precisely, we consider a symmetric and uniformly elliptic coefficient matrix $A\in L^\infty(\R^n;\R^{n\times n})$ with $n \in \mathbb N$ that is periodic 
with respect to the unit cube $Y:=[0,1]^n$.  Setting  $A^\eps(x):=A\left(\frac{x}{\eps}\right)$, we introduce  an energy  functional $E^\eps: H^1(\R^n)  \to \mathbb R$ and  a cost functional $J^\eps:      L^2(\R^n ) \times H^1(\R^n) \to \mathbb R$ as follows:
\begin{equation}\label{energyfunk}
 E^\eps(y):=\int_{\R^n} \nabla y(x)\cdot A^\eps(x)\nabla y(x) + |y(x)|^2\,dx,
\end{equation}
and 
\begin{align}
\label{eq:energyperiodic}
J^\eps(u, y)
:=\frac{\mu_1}{2}E^\eps (y-y^\eps_{d,1})
+\frac{\mu_2}{2}\int_{\R^n} |y(x)-y^\eps_{d,2}(x)|^2\,dx+
\frac{\kappa}{2}\int_{\R^n}|u(x)|^2\,dx
\end{align} 
with  fixed constants   $\mu_1,\mu_2\geq 0$ and $\kappa>0$. Furthermore,  $y^\eps_{d,1} \in H^1(\R^n)$ and $y^\eps_{d,2} \in L^2(\R^n)$ are given functions denoting, respectively, the desired gradient field and the desired state.  In view of \eqref{eq:energyperiodic}, our focus is set on the following linear-quadratic elliptic optimal control problem:
\begin{equation} \tag{$\textnormal{P}_\eps$} \label{P} \left\{
\begin{aligned}
\min  \quad &J^\eps(u_\eps, y_\eps)\\
\textrm{over} \quad   &(u_\eps, y_\eps)  \in U^\eps_{ad} \times H^1(\R^n)\\
\textrm{s.t.} \quad &-\nabla\cdot\left(A^\eps(x)\nabla y_\eps(x)\right) + y_\eps(x)=f^\eps(x) + u_\eps(x)\quad\text{in weak sense on }\R^n.
\
\end{aligned}
\right.
\end{equation}   
In the setting of \eqref{P}, $f^\eps \in L^2(\R^n)$ is a given function, and $\emptyset \neq U^\eps_{ad}\subset L^2(\R^n)$ is a  convex and    closed  subset  representing the  admissible control set. By a classical argument \cite{L,T}, the optimal control problem  \eqref{P} admits a unique optimal solution denoted throughout this paper by $(u^*_\eps, y^*_\eps)  \in U^\eps_{ad} \times H^1(\R^n)$. Let us emphasize that  the admissible control set $U^\eps_{ad}$  and  the  quantities $f^\eps,y^\eps_{d,1},y^\eps_{d,2}$  feature   highly oscillatory structures depending on $\epsilon$.  In particular, the numerical treatment of \eqref{P} is extremely costly since the fine scale, represented by the small parameter $\eps>0$, has to be resolved.  The underlying  $\eps$-dependence  in \eqref{P} will be specified   in Assumptions \ref{ass:data} and \ref{ass:admissibleset}.

More than two decades ago, Kesavan and Saint Jean Paulin \cite{KP1} analyzed the effective behavior of a similar linear-quadratic elliptic optimal control problem. They disregarded the role of the desired state and specified the desired gradient field to be zero, i.e., $\mu_2=0$ and $y^\eps_{d,1}\equiv 0$ in  \eqref{eq:energyperiodic}. However, differently from \eqref{P}, they allowed for a periodic coefficient $B^\eps$ in the   cost functional that may differ from the coefficient $A^\eps$ in  the state equation. The main contribution of  \cite{KP1}  is the weak convergence  for the optimal solution as $\epsilon \to 0$ towards the solution of an effective optimal control problem with an effective state equation. The  effective coefficient  $A^{\text{eff}}$  in the effective state equation corresponds to the classical 
effective coefficient from the periodic homogenization theory, while the effective coefficient $B^{\text{eff}}$ in the energy functional is a perturbation of the classical H-limit of $B^\eps$. 
In the case $A^\eps=B^\eps$, that is relevant for us, one obtains $A^{\text{eff}}=B^{\text{eff}}$. 
Unfortunately,  since the convergence property holds solely in the weak topology,   
the approximation   of the effective optimal control problem could be far  from precise.
Indeed,  from the  classical homogenization theory, it is well known that solutions to 
highly oscillatory problems exhibit   highly oscillatory behavior, and   the homogenization limit provides  an appropriate average of these oscillations. In order to obtain higher-order 
approximations, one has to capture the fast oscillations by correcting the homogenization limit, wherefore results of this type are referred to as corrector results. 
The ultimate goal of this paper is therefore to set up a corrector result of the following form: Given $M\in\N$, we seek for a proper approximation
$(u^*_{M},y^*_{M})$ of  the optimal solution $(u^*_\eps, y^*_\eps)$  to \eqref{P} such that 
\begin{equation}
\label{eq:highorderapp1}
\|u^*_\eps-u^*_{M}\|_{L^2(\R^n)} + \|y^*_\eps-y^*_{M}\|_{L^2(\R^n)}\leq C\eps^{2M}.
\end{equation}
The key tool of our approach is the Bloch wave expansion, which is introduced in the upcoming   section.

\subsection{Bloch expansion}
For the sake of completeness, we briefly motivate and collect well-known facts from the classical Bloch theory in the context of homogenization, 
which can be found, for instance, in \cite{CV,COV1,DLS,santosa}. 

The Bloch expansion is a generalization of the classical Fourier expansion. Every function $u\in L^2(\R^n;\C)$ can be written as 
\begin{equation*}
u(x)=\int_{\R^n}\hat u(\xi)\, e^{2\pi i\xi\cdot x}\,d\xi,
\end{equation*}
where $\hat u(\xi):=\int_{\R^n}u(x)e^{-2\pi i x\cdot \xi}\,dx$ denotes the classical Fourier-transform of $u$. Decomposing every $\xi\in\R^n$ into 
\begin{equation*}
\xi=k+\eta\qquad\text{with }k\in\Z^n\quad\text{and }\eta\in Z:=[-1/2,1/2)^n,
\end{equation*}
one obtains 
\begin{equation*}
u(x)=\int_{Z}\underbrace{\sum_{k\in\Z^n}\hat u(k+\eta)e^{2\pi i k \cdot x}}_{=:F=F(x;\eta)}e^{2\pi i \eta \cdot x }\,d\eta,
\end{equation*}
where for every $\eta\in Z$ the function $F(\cdot,\eta)$ is $Y$-periodic  with $Y:=[0,1]^n$. It can thus be expanded in the $Y$-periodic eigenfunctions to the following eigenvalue problem:

\begin{definition}[Bloch eigenvalue problem] 
Let $\eta\in Z=[-\frac12,\frac12)^{n}$ be fixed. Consider the operator
\begin{equation}
\label{eq:Bloch-Operator}
\mathcal{L}(\eta):=-(\nabla + 2\pi i\eta)\cdot\left(A(\cdot)(\nabla +2\pi i\eta)\right).
\end{equation}
By the classical spectral theory, one finds a sequence of eigenvalues $(\mu_m(\eta))_{m\in\N_0}$ and corresponding $Y$-periodic eigenfunctions
$(\Phi_m(\cdot;\eta))_{m\in\N_0}\subset H^1_\sharp(Y;\mathbb{C})$ of $\mathcal{L}(\eta)$, such that 
\begin{align*}
&0\leq \lambda_0(\eta)\leq \lambda_1(\eta)\leq\dots\quad\text{with }\lambda_m(\eta)\to\infty\,\text{ as }m\to\infty \quad\text{(Bloch eigenvalues)},\\
&(\Phi_m(\cdot;\eta))_{m\in\N_0}\quad\text{form a ONB of } L^2_\sharp(Z;\mathbb{C})\quad\text{(Bloch waves)}.
\end{align*}
The subscript $\sharp$ tags spaces of periodic functions. Furthermore we define
\begin{equation*}
\omega_m(x;\eta):=\Phi_m(x;\eta)e^{2\pi i \eta\cdot x} \quad\text{(quasi-periodic Bloch waves)}.
\end{equation*}
\end{definition}
Note that the quasi-periodic Bloch waves solve the eigenvalue problem
\begin{equation*}
-\nabla\cdot\left(A(x)\nabla\omega_m(x;\eta)\right)=\lambda_m(\eta)\omega_m(x;\eta).
\end{equation*}
This leads to the following classical result. 

\begin{proposition}[Bloch expansion] 
\label{thm:Blochentw}
Every $u\in L^2(\R^n;\mathbb{C})$ admits a unique expansion
\begin{equation}
\label{eq:Blochexpansion}
u(x)=\int_Z\sum_{m\in\N_0} \hat u_m(\eta)\Phi_m(x;\eta)e^{2\pi i \eta\cdot x}\,d\eta=\int_Z\sum_{m\in\N} \hat u_m(\eta)\omega_m(x;\eta)\,d\eta
\end{equation}
with the Bloch coefficients
\begin{equation}
\label{eq:Blochcoefficient}
\hat u_m(\eta):=\int_{\R^n}u(x)e^{-2\pi i \eta\cdot x}\overline{\Phi_m(x;\eta)}\,dx\quad\text{for all }m\in\N_0\,\text{ and }\eta\in Z. 
 \end{equation}
Furthermore Parseval's identity holds
\begin{equation}
\label{eq:Parseval}
\int_{\R^n}|u(x)|^2\,dx=\int_Z\sum_{m\in\N_0}|\hat u_m(\eta)|^2\,d\eta. 
\end{equation}
\end{proposition}

Since the coefficient $A^\eps(x)=A\left(\frac{x}{\eps}\right)$ in \eqref{P} is $\eps$-periodic, one needs to introduce rescaled Bloch waves:
\begin{align}
\label{eq:rescaled1}
\Phi_m^\eps(x;\eta)&:=\Phi_m\left(\frac{x}{\eps};\eps \eta\right),\quad\quad\lambda_m^\eps(\eta):=\frac{1}{\eps^2}\lambda_m(\eps \eta),\\
\label{eq:rescaled2}
\omega_m^\eps(x,\eta)&:=\omega_m\left(\frac{x}{\eps};\eps \eta\right)=\Phi_m^\eps(x;\eta)e^{2\pi i\eta\cdot x}=\Phi_m\left(\frac{x}{\eps};\eps \eta\right)e^{2\pi i\eta\cdot x}.
\end{align} 
The scaling is designed in such a way that 
\begin{equation*}
-\nabla\cdot\left(A^\eps(x)\nabla\omega^\eps_m(x;\eta)\right)=\lambda^\eps_m(\eta)\omega^\eps_m(x;\eta).
\end{equation*}
A simple calculation provides the following rescaled version of the classical Bloch expansion. 

\begin{proposition}[Rescaled Bloch expansion] 
\label{thm:Blochentwrescaled}
Every  $u\in L^2(\R^n;\mathbb{C})$ admits a unique expansion
\begin{align}
\label{eq:RescaledBlochExpansion}
u(x)&=\int_{Z/\eps}\sum_{m\in\N_0} \hat u^\eps_m(\eta)\omega_m^\eps(x;\eta)\,d\eta,\\
\label{eq:RescaledBlochCoefficient}
\hat u_m^\eps(\eta)&=\int_{\R^n}u(x)e^{-2\pi i \eta\cdot x}\overline{\Phi^\eps_m(x;\eta)}\,dx=\int_{\R^n}u(x)\overline{\omega_m^\eps(x;\eta)}\,dx.
\end{align}
Furthermore Parseval's identity holds
\begin{equation}
\label{eq:Plancherel}
\int_{\R^n}|u(x)|^2\,dx=\int_{Z/\eps}\sum_{m\in \N_0}|\hat u_m^\eps(\eta)|^2\,d\eta.
\end{equation}
\end{proposition}
In what follows, we will frequently exploit that due to the eigenvalue property of $\omega_m^\eps$ the operator 
$-\nabla\cdot\left(A^\eps\nabla\right)$ acts on the rescaled Bloch expansion 
as a multiplier. More precisely, for $u$ as in \eqref{eq:RescaledBlochExpansion} one has  
\begin{equation}
\label{eq:multiplierexpansionBloch}
-\nabla\cdot\left(A^\eps(x)\nabla u(x)\right) + u(x)=
\int_{Z/\eps}\sum_{m\in\N_0} (1+\lambda_m^\eps(\eta))\hat u^\eps_{m}(\eta)\omega_m^\eps(x;\eta)\,d\eta
\end{equation}
provided $-\nabla\cdot\left(A^\eps\nabla u\right)\in L^2(\R^n)$. 

It is well known that in the homogenization process only the lowest eigenvalue, the ground state $\lambda_0(\eta)$, is relevant. 
By classical perturbation theory it can be shown that $\lambda_0(\eta)$ is simple and analytic in a neighbourhood of $\eta=0$. Moreover, one immediately finds that $\lambda_0(0)=0$ and 
that $\lambda_0$ is even, wherefore all odd derivatives of $\lambda_0$ vanish. Summing up, one obtains the following Taylor expansion:

\begin{lemma}[Taylor expansion of the first Bloch eigenvalue]  
\label{lem:expansionTaylorlambda0}
For every $k\in\N$ there exists a $2k$-th order tensor $\mathbb{A}^*_{2k}=(\mathbb{A}^*_{2k})_{i_1,\dots ,i_{2k}}$ such that for $M\in\N$ 
\begin{equation}
\label{eq:expansionTaylorlambda0}
\lambda_0(\eta)=(2\pi)^2\mathbb{A}^*_2\cdot\eta^{\otimes 2} + (2\pi)^4\mathbb{A}^*_4\cdot\eta^{\otimes 4}+\dots+(2\pi)^{2M}\mathbb{A}^*_{2M}\cdot\eta^{\otimes 2M} + O\left(|\eta|^{2M+2}\right),
\end{equation} 
where 
\begin{equation*}
\mathbb{A}^*_{2k}\cdot\eta^{\otimes 2k}:=\sum_{i_1,\dots,i_{2k}=1}^n (\mathbb{A}^*_{2k})_{i_1,\dots,i_{2k}}\eta_{i_1}\cdot...\cdot\eta_{i_{2k}}. 
\end{equation*}
\end{lemma}
It is a well-known fact in the Bloch theory that the second order tensor $\mathbb{A}^*_{2}\in\R^{n\times n}$ coincides with the effective coefficient $A^{\text{eff}}$ 
from the classical homogenization theory. As shown, for instance in \cite{ABV} and \cite{DLS2}, the tensors $\mathbb{A}^*_{2k}$ can be computed from periodic cell problems. 
The above expansion   plays a major role in our analysis. 

\subsection{Assumptions}
\label{sec:assumptions}
In this section, we introduce the assumptions for the given data involved  in \eqref{P}, namely, the admissible control set $U^\eps_{ad}\subset L^2(\R^n)$ and 
the functions $f^\eps, y^\eps_{d,1}$, and $y^\eps_{d,2}$. Roughly speaking, we have to demand that the data are well prepared to the periodic microstructure. 

We recall that solutions to homogenization problems are characterized by some effective profile, 
the homogenization limit, and fast lower order oscillations, which adapt the effective profile to the periodic microstructure. These oscillatory corrections can 
be described by periodic cell solutions, the so-called correctors, which are necessary to construct high-order approximations as in \eqref{eq:highorderapp1}.   
In fact, there is an analogue to this adaption process in the Bloch wave approach. As shown in \cite{CV,COV1, DLS}, in a classical homogenization process only the lowest Bloch 
eigenvalue $\lambda^\eps_0$ and the corresponding eigenfunction $\omega^\eps_0$ play a role, while higher Bloch modes are negligible. 
Moreover, it can be shown that the rescaled Bloch eigenfunction $\omega^\eps_0$ is in fact 
an adaption of the Fourier basis function $e^{2\pi i\eta\cdot x}$ to the microstructure (see     \cite{ABV}). In view of this, the following definition is reasonable.
\begin{definition}[Adaption]\label{def:adap}
Let $f\in L^2(\R^n;\C)$ with the Fourier-transform $\hat f$. Then, the adaption $\mathcal{A}^\eps(f)$ of $f$ is defined as follows:
\begin{equation*}
\mathcal{A}^\eps(f):=\int_{Z/\eps}\hat f(\eta)\omega_0^\eps(\cdot;\eta)\,d\eta.
\end{equation*}
\end{definition}
\begin{remark}
In \cite{COV1} $\mathcal{A}^\eps$ is referred to as the Bloch approximation. It is shown that $\mathcal{A}^\eps(f)$ converges almost everywhere to $f$ and, assuming that $f$ is sufficiently smooth with $\hat f$ decaying sufficiently fast at infinity, 
a higher-order asymptotic expansion of $\mathcal{A}^\eps(f)$ holds
\begin{equation}
\label{eq:approxadaption}
\mathcal{A}^\eps(f)(x)=f(x) + \eps\chi^{(1)}\left(\frac{x}{\eps}\right)\cdot\nabla f(x) + \eps^2\chi^{(2)}\left(\frac{x}{\eps}\right)\cdot D^2 f(x) + O(\eps^3),
\end{equation} 
where $\chi^{(1)}, \chi^{(2)}$ are $Y$-periodic corrector functions. 
In particular, the adaption $\mathcal{A}^\eps$ transforms a smooth, non-oscillatory function $f$ to an $\eps$-oscillatory function according to the periodic microstructure. 
\end{remark}

In our approach we will use the fact that the adaption operator $\mathcal{A}^\eps$ preserves the $L^2$-norm when applied to functions that have compact support in Fourier space. 
\begin{lemma}
\label{lem:equalityL2norm}
Let $f\in L^2(\R^n;\C)$ have compact support in Fourier space, i.e., $\supp(\hat f)\subset K$ with $K\subset\R^n$ compact. Let $\eps>0$ be such that $K\subset Z/\eps$. Then 
\begin{equation*}
 \|\mathcal{A}^\eps(f)\|_{L^2(\R^n;\C)}=\|f\|_{L^2(\R^n;\C)}.
\end{equation*}
\end{lemma}

\begin{proof}
Let $\eps>0$ be such that $K\subset Z/\eps$. Then 
\begin{equation*}
\|f\|^2_{L^2(\R^n;\C)}=\int_{\R^n}|\hat f(\eta)|^2\,d\eta=\int_{Z/\eps}|\hat f(\eta)|^2\,d\eta=\|\mathcal{A}^\eps(f)\|^2_{L^2(\R^n;\C)},
\end{equation*}     
where in the second equality we exploited that $\hat f$ is supported on $K\subset Z/\eps$ and in the last equality we used Parseval's identity \eqref{eq:Plancherel} and the definition 
of $\mathcal{A}^\eps$. 
\end{proof}
Let us now state the main assumptions for the given data and    the admissible set $U^\eps_{ad}$ involved in \eqref{P}:
\begin{assumption}
\label{ass:data} 
There exist  $f,y_{d,1},y_{d,2}\in L^2(\R^n)$  and a compact set $K\subset \R^n$ such that 
\begin{align}
\begin{split}
\label{fudwellprepared}
&f^\eps=\mathcal{A}^\eps(f),\quad y^\eps_{d,1}=\mathcal{A}^\eps(y_{d,1})\quad\text{and }y^\eps_{d,2}=\mathcal{A}^\eps (y_{d,2}),\\
&\supp(\hat f),\, \supp(\hat y_{d,1}),\, \supp(\hat y_{d,2})\subset K.
\end{split}
\end{align}
Moreover,  the compact set $K$ is assumed to be symmetric in the sense that 
\begin{equation} \label{ass:sym}
\eta\in K\,\Rightarrow -\eta\in K.
\end{equation}
\end{assumption}

\begin{remark}
\begin{itemize}
\item[]
\item[(i)] The condition that the Fourier transforms  $\hat f, \hat y_{d,1}, \hat y_{d,2}$ have compact support implies readily  that  $f,y_{d,1}, y_{d,2}\in H^k(\R^n)$ for all $k\in \N$. 
Therefore, along with \eqref{ass:sym}, it follows that their adaptions $f^\eps,y^\eps_{d,1},y^\eps_{d,2}$ are real-valued and belong to $H^k(\R^n)$ for all $k\in \N$, cf.  \cite{CV,COV1}. 
Also, notice that  Assumption \ref{ass:data} demands that $f^\eps, y^\eps_{d,1}$ and $y^\eps_{d,2}$ are highly oscillatory according to the micro-structure. 
Actually this assumption is quite natural as solutions to $\eps$-periodic coefficient equations such as the state equation in \eqref{P} 
typically exhibit a highly oscillatory behavior.  
\item[(ii)]
Relation \eqref{eq:approxadaption} suggests that the difference $\mathcal{A^\eps}(f)-f$ is typically of order $\eps$. For unprepared data, i.e.,
$y^\eps_{d,1}=y_{d,1}, y^\eps_{d,2}=y_{d,2}$, and $f^\eps=f$, the effective model from Definition \ref{def:highordereffectiveproblem} below is therefore expected to approximate 
the original problem only up to errors of order $\eps$.  In fact, well-prepared data are necessary for many corrector results in the homogenization theory. See, for instance, 
\cite{ALR,BG, BFM} in the context of periodic wave equations. 
\end{itemize}
\end{remark}
We now state the assumption on the admissible set $U^\eps_{ad}$. 
\begin{assumption}
\label{ass:admissibleset} Let $K\subset \R^n$ be as in Assumption \ref{ass:data}.
For the admissible set $U^\eps_{ad}\subset L^2(\R^n)$, we assume the following:
\begin{itemize}
 \item[(i)] $U^\eps_{ad}$ is  convex and  closed.  
 \item[(ii)] There exists $u_0\in L^2(\R^n)$ such that for every $\eps>0$ one has $\mathcal{A}^\eps(u_0)\in U^\eps_{ad}$ .  
 \item[(iii)] For every  $u\in U^\eps_{ad}$ with 
Bloch expansion
\begin{equation*}
u(x)=\int_{Z/\eps}\sum_{m\in\N_0} \hat u^\eps_m(\eta)\omega_m^\eps(x;\eta)\,d\eta,
\end{equation*}
it follows that the projection
\begin{equation}
\label{eq:restrictionK}
u_K(x):=\int_K\hat u^\eps_0(\eta)\omega_0^\eps(x;\eta)\,d\eta
\end{equation}
lies as well in $U^\eps_{ad}$. In other words, $(U^\eps_{ad})_K:= \{ u_K\ | \ u  \in  U^\eps_{ad} \} \subset U^\eps_{ad}.$
\end{itemize}
\end{assumption}
\begin{remark}  The second condition $(ii)$ guarantees that      $(u_\eps^*)_{\epsilon>0} \subset  L^2(\R^n)$   is  uniformly bounded  as shown in  Lemma \ref{lem:assumption2}. 
The third condition $(iii)$ demands that for every $\eps>0$ the projection $(U^\eps_{ad})_K$ of $U^\eps_{ad}$ in Bloch space to the lowest mode $m=0$ and 
$\eta\in K$ is a subset of $U^\eps_{ad}$.  
\end{remark}

For the rest of the article, $\eps>0$ is always assumed to be sufficiently small such that $K\subset Z/\eps$.   Furthermore      $S_\epsilon: L^2(\R^n) \to H^1(\R^n)$ denotes the control-to-state operator associated with \eqref{P}.

\begin{lemma}
\label{lem:assumption2}
 Let Assumptions \ref{ass:data} and  \ref{ass:admissibleset} hold. Then,     $(u^*_\eps)_{\eps>0} \subset L^2(\R^n)$  
 is  uniformly  bounded, i.e., there exists a     constant $C>0$, independent of $\eps$, such that 
 \begin{equation*}
  \|u^*_\eps\|_{L^2(\R^n)}    \leq C \quad \forall \epsilon >0.
 \end{equation*} 
\end{lemma}
\begin{proof}According to Assumption \ref{ass:admissibleset}, there exists 
an element $u_0\in L^2(\R^n)$   such that $\mathcal{A}^\eps(u_0)\in U^\eps_{ad}$.    For this reason,
\begin{equation}
\label{eq:L2estimatecontrols}
 \frac{\kappa}{2}\|u^*_\eps\|^2_{L^2(\R^n)}\leq J^\eps(u^*_\eps, S_\epsilon u_\eps^*)\leq J^\eps(\mathcal{A}^\eps(u_0),\underbrace{S_\epsilon\mathcal{A}^\eps(u_0)}_{=: y^\eps_0}) \quad \forall \epsilon >0. 
\end{equation}
Our aim is to find a uniform bound for the right hand side of \eqref{eq:L2estimatecontrols}. 

Let $\hat u_0$ be the Fourier transform of $u_0$. By construction of $\mathcal{A^\eps}$ (Definition \ref{def:adap}), Proposition \ref{thm:Blochentwrescaled} yields that the lowest-order Bloch-coefficient ($m=0$) of $\mathcal{A}^\eps(u_0)$ 
corresponds to the Fourier transform $\hat u_0$, while all higher-order Bloch modes with $m\geq 1$ vanish. In the following, let $C>0$ denote a  $\eps$-independent generic constant
that can vary from line to line.
By Lemma \ref{lem:EnergyinBlochform} below, the right-hand side in    \eqref{eq:L2estimatecontrols} can be written as
{\small\begin{align*}
J^\eps(\mathcal{A}^\eps(u_0),y^\eps_0)= \frac{\mu_1}{2} \int_{K}\left|\frac{\hat f(\eta) + \hat u_{0}(\eta)}{1+\lambda_0^\eps(\eta)}-\hat y_{d,1}(\eta)\right|^2\!\!\!\!(1+\lambda^\eps_0(\eta))\,d\eta
+ \frac{\mu_1}{2}\int_{(Z/\eps)\setminus K}\frac{| \hat u_{0}(\eta)|^2}{1+\lambda_0^\eps(\eta)}\,d\eta\\
+\frac{\mu_2}{2} \int_{K}\left|\frac{\hat f(\eta) + \hat u_{0}(\eta)}{1+\lambda_0^\eps(\eta)}-\hat y_{d,2}(\eta)\right|^2\,d\eta
+\frac{\mu_2}{2} \int_{(Z/\eps)\setminus K}\frac{|\hat u_{0}(\eta)|^2}{(1+\lambda_0^\eps(\eta))^2}\,d\eta +  \frac{\kappa}{2}\int_{Z/\eps} |\hat u_0(\eta)|^2\,d\eta, \\
\leq C\left(\int_{Z/\eps} |\hat u_0(\eta)|^2\,d\eta + \int_K |\hat f(\eta)|^2+|\hat y_{d,1}(\eta)|^2(1+\lambda_0^\eps(\eta))+|\hat y_{d,2}(\eta)|^2\,d\eta\right)\\
\leq C\left(\|u_0\|^2_{L^2(\R^n)} + \|f\|^2_{L^2(\R^n)} + \|y_{d,1}\|^2_{H^1(\R^n)}+ \|y_{d,2}\|^2_{L^2(\R^n)}\right),
\end{align*}
}
where in the first inequality we used that $\lambda_0^\eps\geq 0$. In the second inequality we used  Parseval's identity \eqref{eq:Plancherel} and exploited the fact that $\int_K|\hat y_{d,1}(\eta)|^2(1+\lambda_0^\eps(\eta))\,d\eta$ is 
controlled by the $H^1$-norm of $y_{d,1}$ (see \cite{COV1} for details). 
\end{proof}

\begin{example}	
\label{ex:admissiblesets}  Let Assumption \ref{ass:data} hold. Then 
the following admissible control sets satisfy Assumption \ref{ass:admissibleset}.   
\begin{itemize}
\item[(i)] Full space $U^\eps_{ad}=L^2(\R^n)$. 
\item[(ii)] Bounds on the $L^2$-norm of the control: For $L>0$ let
\begin{equation*}
U^\eps_{ad}:=\{u\in L^2(\R^n)\,|\,
\|u\|_{L^2(\R^n)}\leq L\}.
\end{equation*}
It is obvious that  $U^\eps_{ad} \subset  L^2(\R^n)$ is convex and closed.  The condition (ii) of Assumption \ref{ass:admissibleset} is satisfied with $u_0=0$. The third condition   is a direct consequence of   \eqref{eq:restrictionK} and Parseval's identity \eqref{eq:Plancherel}.  \item[(iii)] Bounds on the $L^2$-norm of the state: For $L>0$ let   
\begin{equation*}
U^\eps_{ad}:=\left\{u\in L^2(\R^n)\,\Big|\,\|S_\epsilon u\|_{L^2(\R^n)}\leq L\right\}.
\end{equation*}
As the control-to-state operator $S_\epsilon: L^2(\R^n) \to H^1(\R^n)$ is affine linear and continuous, the set  $U^\eps_{ad} \subset  L^2(\R^n)$ is convex and closed. 
The condition  (ii) of Assumption \ref{ass:admissibleset} is satisfied with $u_0=-f$ since according to Assumption \ref{ass:data} 
\begin{equation}\label{exm3}
S_\epsilon \left(\mathcal{A^\eps}(-f)\right)= S_\epsilon(- f^\eps)  = 0 \quad \Rightarrow \quad \mathcal{A^\eps}(-f)\in U^\eps_{ad}  \quad \forall   \epsilon >0.
\end{equation}
The third condition of Assumption \ref{ass:admissibleset}  follows from the formula \eqref{eq:expansionyeps} and Parseval's identity \eqref{eq:Plancherel}  implying      $ \|S_\epsilon u_K\|_{L^2(\R^n)} \le \|S_\epsilon u\|_{L^2(\R^n)} \le L$ for all $u \in U_{ad}$.

\item[(iv)] Bounds on the energy of the state: For $L>0$ let  
\begin{equation*}
U^\eps_{ad}:=\left\{u\in L^2(\R^n)\,\Big|E^\eps(S_\epsilon u)\leq L\right\}.
\end{equation*}
Since by \eqref{energyfunk} the energy functional $E^\eps: H^1(\mathbb R^n) \to \R$  defines a norm on $H^1(\mathbb R^n)$, the set  $U^\eps_{ad} \subset  L^2(\R^n)$ is again convex and closed.  
Thanks to \eqref{exm3},   the choice $u_0=-f$ satisfies again the condition (ii) of  Assumption \ref{ass:admissibleset}.
 In view of \eqref{eq:energyfirstpart} with $\hat y_{d,1}\equiv0$, the third condition $(U^\eps_{ad})_K\subset U^\eps_{ad}$ is again satisfied by Parseval's identity \eqref{eq:Plancherel}.  

\end{itemize}
\end{example}
Note that pointwise conditions on $u$ or $y$ are not covered by our approach. However, since solutions to homogenization problems typically 
exhibit a highly oscillatory behavior, pointwise conditions are not appropriate in the framework of high-order homogenization.   
 
\subsection{Outline and main results}
In the following we will expand the optimal solution $(u_\eps^*,y_\eps^*)$ of \eqref{P} in rescaled Bloch waves 
according to Proposition \ref{thm:Blochentwrescaled}. Our analysis consists of four steps:
In the first step, see Proposition \ref{prop:representationasadaption}, we find that the optimal solution $(u_\eps^*,y_\eps^*)$ is in fact an adaption: 
\begin{equation*}
u_\eps^*=\mathcal{A}^\eps(\tilde u^*_\eps)\quad\text{ and }\quad y_\eps^*=\mathcal{A}^\eps(\tilde y^*_\eps) 
\end{equation*}
for some appropriate  pair $(\tilde u^*_\eps,\tilde y^*_\eps)$. 
In the second step, we simplify the (Fourier) expansions of $\tilde u^*_\eps$ and $\tilde y^*_\eps$ using the analyticity of the lowest Bloch eigenvalue. 
Based on this first approximation, in our third step, we derive an $M$-th order effective optimal control problem of the following form. 

\begin{definition}[High-order effective optimal control problem] 
\label{def:highordereffectiveproblem}
  Let Assumptions \ref{ass:data} and  \ref{ass:admissibleset} hold. Then, for every $M \in \N$, we define the high-order effective cost functional $J_M^*: L^2(\R^n) \times H^M(\R^n) \to \R$ as follows:
\begin{align}
\begin{split}
\label{eq:energyeffectiveM}
J_M^*(  u,   y):=&\frac{\mu_1}{2} E^*_M(  y-y_{d,1})
+\frac{\mu_2}{2}\int_{\R^n}|  y(x)-y_{d,2}(x)|^2\,dx +\frac{\kappa}{2}\int_{\R^n}|  u(x)|^2\,dx, 
\end{split}
\end{align}
\begin{equation}
 \label{eq:energyeffective equation}
 E^*_M(y):=\int_{\R^n}\left(\sum_{k=1}^M\eps^{2k-2}\left\langle D^{k}y(x), D^{k} y(x),
\right\rangle_{\mathbb{A}^*_{2k}} + |y(x)|^2\right)\,dx,
\end{equation}
where   the tensors $\mathbb{A}^*_{2k}$ are as in Lemma \ref{lem:expansionTaylorlambda0} and 
\begin{equation*}
\left\langle D^{k}y(x), D^{k}y(x)
\right\rangle_{\mathbb{A}^*_{2k}}:=
\sum_{i_1,\dots,i_{2k}=1}^n (\mathbb{A}^*_{2k})_{i_1,\dots,i_{2k}}\partial^{k}_{x_{i_1}x_{i_2}\dots x_{i_{k}}}\,y(x)\, 
\partial^{k}_{x_{i_{k+1}}x_{i_{k+2}}\dots x_{i_{2k}}}\,y(x).
\end{equation*}
Introducing   $\mathbb{A}^*_{2k}D^{2k}:=\sum_{i_1,\dots,i_{2k}=1}^n (\mathbb{A}^*_{2k})_{i_1,\dots,i_{2k}}\partial^{2k}_{x_{i_1}x_{i_2}\dots x_{i_{2k}}}$, we consider the   effective state equation
\begin{equation}
\label{eq:Mthordereffectiveequation}
\sum_{k=1}^M\eps^{2k-2}(-1)^{k}\mathbb{A}^*_{2k}D^{2k}  y(x) +   y(x)
=f(x)+   u(x)
\end{equation} 
and propose the following high-order effective optimal control problem:
\begin{equation} \tag{$\textnormal{P}^*_M$} \label{PM} \left\{
\begin{aligned}
\min  \quad &J^*_M(  u,   y)\\
\textnormal{over} \quad   &(  u,    y)  \in U^*_{ad} \times H^{2M}(\R^n)\\
\textnormal{s.t.} \quad &\eqref{eq:Mthordereffectiveequation},
\end{aligned}
\right.
\end{equation}
 where the effective admissible control set $U^*_{ad}$ is given by the inverse image of $U^\eps_{ad}$ 
under the adaption operator $\mathcal{A}^\eps$ projected onto $K$, 
\begin{equation}
\label{eq.effectiveadmissibleset} 
U^*_{ad}:=\left\{  u(x)=\int_K \hat u(\eta)e^{2\pi i\eta\cdot x}\,d\eta\,|\, \mathcal{A}^\eps(  u)\in U^\eps_{ad} \right\}.
\end{equation}
\end{definition} 
In the fourth and final step, we show that  \eqref{PM} admits for all sufficiently small $\eps>0$ a unique optimal solution $(  u^*_M,    y^*_M)  \in U^*_{ad} \times H^{2M}(\R^n)$, which relies on a reformulation of the problem 
in Fourier space (see Proposition \ref{prop:reformulationoptimalcontrolBlochFourier}). As a main result, 
we prove in Proposition \ref{prop:approximationPDEconstraint} and Corollary \ref{cor:errorcontrols} that the   optimal solution $( {u}^*_{M}, {y}^*_{M})$ to \eqref{PM} satisfies the error estimate
\begin{equation}
\label{eq:approximationerrorcontrols}
\|\tilde u^*_\eps-  u^*_{M}\|_{L^2(\R^n)} + \|\tilde y^*_\eps-  y^*_{M}\|_{L^2(\R^n)}\leq C\eps^{2M}
\end{equation}
 for some $\eps$-independent constant $C>0$. We recall that the optimal solution to \eqref{P}   satisfies $u_\eps^*=\mathcal{A}^\eps(\tilde u^*_\eps)$ 
and  $y_\eps^*=\mathcal{A}^\eps(\tilde y^*_\eps)$. Since the functions $\tilde u^*_\eps, u^*_{M}, \tilde y^*_\eps$, and $ y^*_{M}$ have compact support in Fourier space, Lemma \ref{lem:equalityL2norm} and \eqref{eq:approximationerrorcontrols} yield our final result:

\begin{corollary} For all sufficiently small $\epsilon >0$, it holds that 
\begin{equation}
\| u^*_\eps-\mathcal{A}^\eps(  u^*_{M})\|_{L^2(\R^n)} + \|y^*_\eps-\mathcal{A}^\eps(  y^*_{M})\|_{L^2(\R^n)}\leq C\eps^{2M}.
\end{equation}
\end{corollary}
   
Let us comment on the high-order effective optimal control problem \eqref{PM}.  
This optimization problem still depends on $\eps$. It can nevertheless be regarded as an effective model in the sense of  the homogenization theory since the coefficients $\mathbb{A}^*_{2k}$ 
in the effective cost functional $J_M^*$ and the effective equation \eqref{eq:Mthordereffectiveequation} are $\eps$-independent. 
Note that, for $M=1$,  \eqref{PM} coincides with the classical 
homogenization limit   obtained in \cite{KP1}. For $M>1$ it can be understood as a higher-order approximation. 
The effective admissible set $U^*_{ad}$ may depend on $\eps$. In Section \ref{subsec:wellposedness} we 
determine $U^*_{ad}$ corresponding to the admissible sets from Example \ref{ex:admissiblesets}.  

\begin{remark}
As shown in Section \ref{subsec:wellposedness} below the effective state equation \eqref{eq:Mthordereffectiveequation} is well-posed for $\eps>0$ lying below a certain threshold $\eps_M$. 
Such a threshold can only be found if the right hand side of \eqref{eq:Mthordereffectiveequation} has compact support in Fourier space. This is why the effective admissible set \eqref{eq.effectiveadmissibleset} 
is restricted to functions with compact support $K$ in Fourier space. In Section \ref{sec:Wellposedepsindependent} we provide an alternative $M$-th order effective optimal control problem 
which is well-posed independently of $\eps$ and which has the same approximation quality as the one from Defintion \ref{def:highordereffectiveproblem}. 
In that situation the effective admissible set may be defined as $\left(\mathcal{A}^\eps\right)^{-1}(U^\eps_{ad})$. 
However, the effective optimal control still 
lies in the projected set $U^*_{ad}$ from \eqref{eq.effectiveadmissibleset}. 
\end{remark}
The rest of of the paper is organized as follows. In Section \ref{subsec:formalderivation} we derive and justify the effective model  \eqref{PM}, 
and in Section \ref{subsec:wellposedness} we show that it admits a unique optimal solution for all  sufficiently small $\eps>0$. Section \ref{sec:errorestimates} provides the central error estimates and proves the approximation property 
of  \eqref{PM}. In this context a high-order approximation for the adjoint state is obtained. In Section \ref{sec:Wellposedepsindependent} 
an alternative effective problem is proposed, which is   well-posed independently of $\eps$ and  related to  \eqref{PM} through an algebraic manipulation.

\section{The effective optimal control problem}
In this section, we justify the $M$-th order effective optimal control problem \eqref{PM}. The key steps in the derivation are 
(1) rewriting the original problem in Bloch space and (2) expanding the lowest Bloch eigenvalue in a Taylor series (see Section \ref{subsec:formalderivation}). 
In Section \ref{subsec:wellposedness} we  prove that for  all sufficiently small $\eps>0$, \eqref{PM} admits a unique optimal solution. 
\subsection{Derivation of the effective model}
\label{subsec:formalderivation}
As the first step, using Proposition \ref{thm:Blochentwrescaled}, we expand the optimal solution 
$(u^*_\eps, y_\eps^*)$ to  \eqref{P} in the rescaled Bloch waves as follows:
\begin{align*}
u^*_\eps(x)&=\int_{Z/\eps}\sum_{m\in\N_0} \hat u^*_{\eps,m}(\eta)\omega_m^\eps(x;\eta)\,d\eta\quad \textrm{and} \quad
y_\eps^*(x)&=\int_{Z/\eps}\sum_{m\in\N_0} \hat y^*_{\eps,m}(\eta)\omega_m^\eps(x;\eta)\,d\eta.
\end{align*} 
We then specify the Bloch expansion of the optimal state $y_\eps^*(x)$ and rewrite the cost functional $J^\eps$ in terms of the Bloch transform.
\begin{lemma}
\label{lem:EnergyinBlochform}
Let   Assumptions  \ref{ass:data} and  \ref{ass:admissibleset} hold. Then,
for every control $u_\eps\in U^\eps_{ad}$   with Bloch coefficients $\hat u_{\eps,m}$, the corresponding state  $y_\eps:= S_\epsilon u_\epsilon$ satisfies the following expansion:
\begin{align}
\begin{split}
\label{eq:expansionyeps}
 y_\eps(x)&=\int_{K}\frac{\hat f(\eta) +\hat u_{\eps,0}(\eta)}{1+\lambda_0^\eps(\eta)}\omega_0^\eps(x;\eta)\,d\eta
+\int_{(Z/\eps)\setminus K}\frac{\hat u_{\eps,0}(\eta)}{1+\lambda_0^\eps(\eta)}\omega_0^\eps(x;\eta)\,d\eta\\
&+\int_{Z/\eps}\sum_{m\geq 1}\frac{\hat u_{\eps,m}(\eta)}{1+\lambda_m^\eps(\eta)}\omega_m^\eps(x;\eta)\,d\eta,
\end{split}
\end{align}
and the cost functional $J^\eps$ is given by
\begin{equation*}
J^\eps(u_\eps, y_\eps)=
\hat J^\eps(\hat u_{\eps,m}):=\frac{\mu_1}{2}\hat J^\eps_1(\hat u_{\eps,m}) + \frac{\mu_2}{2}\hat J^\eps_2(\hat u_{\eps,m}) + \frac{\kappa}{2}\int_{Z/\eps}\sum_{m\in\N_0}|\hat u_{\eps,m}(\eta)|^2\,d\eta, 
\end{equation*}
where
\begin{align}
\hat J^\eps_1(\hat u_{\eps,m})=&\int_{K}\left|\frac{\hat f(\eta) + \hat u_{\eps,0}(\eta)}{1+\lambda_0^\eps(\eta)}-\hat y_{d,1}(\eta)\right|^2(1+\lambda^\eps_0)\,d\eta
+\int_{(Z/\eps)\setminus K}\frac{| \hat u_{\eps,0}(\eta)|^2}{1+\lambda_0^\eps(\eta)}\,d\eta\nonumber\\
\label{eq:energyfirstpart}
+&\int_{Z/\eps}\sum_{m\geq 1}\frac{|\hat u_{\eps,m}(\eta)|^2}{1+\lambda_m^\eps(\eta)}\,d\eta,\\
\hat J^\eps_2(\hat u_{\eps,m})=&\int_{K}\left|\frac{\hat f(\eta) + \hat u_{\eps,0}(\eta)}{1+\lambda_0^\eps(\eta)}-\hat y_{d,2}(\eta)\right|^2\,d\eta
+\int_{(Z/\eps)\setminus K}\frac{|\hat u_{\eps,0}(\eta)|^2}{(1+\lambda_0^\eps(\eta))^2}\,d\eta\nonumber\\
\label{eq:energysecondpart}
+&\int_{Z/\eps}\sum_{m\geq 1}\frac{|\hat u_{\eps,m}(\eta)|^2}{(1+\lambda_m^\eps(\eta))^2}\,d\eta.
\end{align}
 \end{lemma}

\begin{proof}Thanks to \eqref{fudwellprepared}, the Bloch coefficients of 
$f^\eps=\mathcal{A}^\eps(f),  y^\eps_{d,1}=\mathcal{A}^\eps(y_{d,1}),  y^\eps_{d,2}=\mathcal{A}^\eps (y_{d,2})$ vanish for all modes $m\in \N$, and for  $m=0$ they are given by 
\begin{equation} \label{eq:Blochgivenfunc0}
\hat f_{\eps,0} = \left\{\begin{aligned}
  &\hat f  &\textrm{in }  K\\
  &0 &\textrm{in } (Z/\eps)\setminus K, 
  \end{aligned}
  \right. \
  \hat y^\eps_{d,1,0} = \left\{\begin{aligned}
  &\hat y^\eps_{d,1} &\textrm{in }  K\\
  &0 &\textrm{in } (Z/\eps)\setminus K, 
  \end{aligned}
  \right.\
  \hat  y^\eps_{d,2,0} = \left\{\begin{aligned}
  &\hat y^\eps_{d,2}  &\textrm{in }  K\\
  &0 &\textrm{in } (Z/\eps)\setminus K.
  \end{aligned}
  \right.
\end{equation}
Now, the first claim \eqref{eq:expansionyeps} follows directly from   \eqref{eq:multiplierexpansionBloch} applied to $y_\eps$, 
\begin{equation}
\label{eq:operatortoBloch}
-\nabla\cdot\left(A^\eps(x)\nabla y_\eps(x)\right) + y_\eps(x)=
\int_{Z/\eps}\sum_{m\in\N_0} (1+\lambda_m^\eps(\eta))\hat y_{\eps,m}(\eta)\omega_m^\eps(x;\eta)\,d\eta.
\end{equation}
By comparing this expansion with the Bloch expansion of $f^\eps+u_\eps$ along with \eqref{eq:Blochgivenfunc0},  the claim \eqref{eq:expansionyeps} follows. In view of \eqref{eq:expansionyeps}, the Bloch coefficients of $y_\eps$ are given by
\begin{equation} \label{eq:Blochylemma0}
 \hat y_{\eps,m}  = \frac{\hat u_{\eps,m}(\eta)}{1+\lambda_m^\eps(\eta)} \quad \forall m \in \N \quad \textrm{and} \quad
\hat y_{\eps,0} = \left\{\begin{aligned}
  &\frac{\hat f(\eta) +\hat u_{\eps,0}(\eta)}{1+\lambda_0^\eps(\eta)}   &\textrm{in }  K\\
  &\frac{\hat u_{\eps,0}(\eta)}{1+\lambda_0^\eps(\eta)}   &\textrm{in } (Z/\eps)\setminus K.
  \end{aligned}
  \right.
 \quad \end{equation}
Making use of \eqref{eq:Blochylemma0} and \eqref{eq:Blochgivenfunc0}, the claim \eqref{eq:energysecondpart} follows   directly from   Parseval's identity \eqref{eq:Plancherel} applied to  $y_\epsilon-y^\eps_{d,2}.$ Finally, to prove \eqref{eq:energyfirstpart}, we   apply \eqref{eq:multiplierexpansionBloch}  to $y_\eps-y^\eps_{d,1}$ and obtain that
\begin{equation} \label{eq:Blochylemma1}
\begin{aligned}
&-\nabla\cdot\left(A^\eps(x)\nabla(y_\eps-y^\eps_{d,1})(x)\right) + (y_\eps-y^\eps_{d,1})(x)\\
=&
\int_K (1+\lambda^\eps_0(\eta))(\hat y_{\eps,0}-\hat y_{d,1})(\eta)\omega_0^\eps(x;\eta)\,d\eta
+\int_{(Z/\eps)\setminus K} \!\!\!\! \!\!\!\!(1+\lambda^\eps_0(\eta))\hat y_{\eps,0}(\eta)\omega_0^\eps(x;\eta)\,d\eta\\
&+\int_{Z/\eps}\sum_{m\geq 1}(1+\lambda^\eps_m(\eta))\hat y_{\eps,m}(\eta)\omega_m^\eps(x;\eta)\,d\eta.
\end{aligned}
\end{equation}
Applying the Plancherel  identity for the Bloch expansion to the scalar product $(-\nabla\cdot\left(A^\eps \nabla(y_\eps-y^\eps_{d,1})\right) + y_\eps-y^\eps_{d,1}, y_\eps-y^\eps_{d,1})_{L^2(\R^n)}$ and the scalar product of their respective Bloch representations given by \eqref{eq:Blochylemma1}, \eqref{eq:Blochylemma0}, and \eqref{eq:Blochgivenfunc0}, we  conclude that   \eqref{eq:energyfirstpart} is valid.  
 \end{proof}
In the next step, we show that for the optimal control $u^*_\eps$ all Bloch modes with $m\geq 1$ or $\eta\in (Z/\eps)\setminus K$ vanish. 

\begin{proposition}
\label{prop:representationasadaption}
Under the assumptions of Lemma \ref{lem:EnergyinBlochform}, it holds that
\begin{align}
\label{eq:formulaueps}
u^*_\eps(x)&=\int_{K}\hat u^*_{\eps,0}(\eta)\omega_0^\eps(x;\eta)\,d\eta,\\
\label{eq:formulayeps}
y_\eps^*(x)&=\int_{K}\frac{\hat f(\eta) + \hat u^*_{\eps,0}(\eta)}{1+\lambda_0^\eps(\eta)}\omega_0^\eps(x;\eta)\,d\eta.
\end{align}
In particular, $u^*_\eps=\mathcal{A}^\eps(\tilde u^*_\eps)$ and $y_\eps^*(x)=\mathcal{A}^\eps(\tilde y^*_\eps)$ with 
\begin{align}
\label{eq:defutildeeps}
\tilde u^*_\eps(x)&=\int_K\hat u^*_{\eps,0}(\eta)e^{2\pi i\eta\cdot x}\,d\eta,\\
\label{eq:defytildeeps}
\tilde y_\eps^*(x)&=\int_{K}\frac{\hat f(\eta) + \hat u^*_{\eps,0}(\eta)}{1+\lambda_0^\eps(\eta)}e^{2\pi i\eta\cdot x}\,d\eta.
\end{align} 
\end{proposition}

\begin{proof}
Our aim is to show that $\hat u^*_{\eps,m}(\eta)=0$ for $m\geq 1$ or $\eta\in (Z/\eps)\setminus K$. 
Indeed, from the representation formula for $J^\eps$  (Lemma \ref{lem:EnergyinBlochform}), one directly concludes that for 
$u\in U^\eps_{ad}$ the projection $u_K$ in the sense of \eqref{eq:restrictionK} satisfies
\begin{equation*}
J^\eps(u,S_\epsilon u)\geq J^\eps(u_K, S_\epsilon u_K)
\end{equation*}
with equality iff $u=u_K$.
Since by Assumption \ref{ass:admissibleset} the projection $u_K$ of  every $u\in U^\eps_{ad}$ lies again in $U^\eps_{ad}$, 
the (unique) optimal control  satisfies $u^*_\eps=(u^*_\eps)_K$, which is the claimed result \eqref{eq:formulaueps}.  
The statement \eqref{eq:formulayeps} is then a direct consequence of \eqref{eq:formulaueps} and \eqref{eq:expansionyeps}. 
\end{proof}

  Proposition \ref{prop:representationasadaption} states that the optimal control of \eqref{P} is an adaption $u^*_\eps=\mathcal{A}^\eps(\tilde u^*_\eps)$, with the Bloch coefficient 
$\hat u^*_{\eps,0}$ supported in $K$.    
In order to construct a first approximation for $y^*_\eps$, we exploit the formula \eqref{eq:formulayeps} and the fact that the lowest Bloch eigenvalue $\lambda_0$ 
is analytic in a neighbourhood of $\eta=0$ (cf. Lemma \ref{lem:expansionTaylorlambda0}). In particular, for $M\in \N$ arbitrary, we can expand
\begin{align}
\begin{split}
\label{eq:expansionlambdaeps}
\lambda_0^\eps(\eta)=\frac{1}{\eps^2}\lambda_0(\eps\eta)
=&(2\pi)^2\mathbb{A}^*_2\cdot\eta^{\otimes 2} + \eps^2(2\pi)^4\mathbb{A}^*_4\cdot\eta^{\otimes 4}+\dots\\
&+\eps^{2M-2}(2\pi)^{2M}\mathbb{A}^*_{2M}\cdot\eta^{\otimes 2M} + O\left(\eps^{2M}|\eta|^{2M+2}\right),
\end{split}
\end{align} 
where the error is of order $\eps^{2M}$ uniformly in $\eta\in K$. With  
\begin{equation}
\label{eq:polynomialBloch}
P_M^\eps(\eta):=(2\pi)^2\mathbb{A}^*_2\cdot\eta^{\otimes 2} + \eps^2(2\pi)^4\mathbb{A}^*_4\cdot\eta^{\otimes 4}+\dots+\eps^{2M-2}(2\pi)^{2M}\mathbb{A}^*_{2M}\cdot\eta^{\otimes 2M},
\end{equation}
we propose   the following approximation for $\tilde y^*_\eps$:
\begin{equation}
\label{eq:ystarMappdef}
\tilde  y^*_{M,app}(x):=\int_{K}\frac{\hat f(\eta) + \hat u^*_{\eps,0}(\eta)}{1+P_M^\eps(\eta)}e^{2\pi i\eta\cdot x}\,d\eta.
\end{equation}
Indeed, the above function $\tilde y^*_{M,app}$ provides a good approximation for $\tilde y^*_\eps$ in the following sense. 

\begin{proposition}
\label{prop:firstapproximationy}
Let $\eps>0$ be such that $1+P^\eps_M(\eta)>0$ for all $\eta\in K$. 
Let $\tilde y^*_\eps$ be as in \eqref{eq:defytildeeps}, and let $\tilde y^*_{M,app}$ be defined through \eqref{eq:ystarMappdef}. Then
there exists a constant $C>0$ such that 
\begin{equation*}
\|\tilde y^*_\eps-\tilde y^*_{M,app}\|_{L^2(\R^n)}\leq C\eps^{2M}.
\end{equation*}
\end{proposition}

\begin{proof}
Using Parseval's identity for Fourier expansions one finds
\begin{align*}
\|\tilde y^*_\eps-\tilde y^*_{M,app}\|^2_{L^2(\R^n)}
=\int_K|\hat f(\eta) + \hat u^*_{\eps,0}(\eta)|^2\left|\frac{1}{1+\lambda_0^\eps(\eta)}
-\frac{1}{1+P_M^\eps(\eta)}\right|^2\,d\eta\\
\leq\left\|\frac{1}{1+\lambda_0^\eps(\eta)}
-\frac{1}{1+P_M^\eps(\eta)}\right\|^2_{L^\infty(K)}\int_{K}|\hat f(\eta) + \hat u^*_{\eps,0}(\eta)|^2\,d\eta\\
\leq \tilde C\eps^{4M}\|\hat f + \hat u^*_{\eps,0}\|^2_{L^2(K)} \underbrace{\leq}_{\eqref{eq:Plancherel}} 2\tilde C\eps^{4M}\left(\|f\|^2_{L^2(\R^n)} + \|u^*_{\eps}\|^2_{L^2(\R^n)}\right)\underbrace{\leq}_{\textnormal{Lemma} \, \ref{lem:assumption2}} C^2\eps^{4M},
\end{align*}
where for the second inequality we exploited that the error between $P_M^\eps$ and $\lambda_0^\eps$ is of order $\eps^{2M}$ uniformly in $\eta\in K$. 
\end{proof}

By a direct calculation,   the approximation $\tilde y^*_{M,app}$ is a solution to the following $(2M)$-th order constant coefficient equation
\begin{equation}
\label{eq:wrongeffectivesolutionType2}
\sum_{k=1}^M\eps^{2k-2}(-1)^{k}\mathbb{A}^*_{2k}D^{2k}\tilde y^*_{M,app}(x) + \tilde y^*_{M,app}(x)
=f(x)+ \tilde u^*_\eps(x),
\end{equation} 
where $\tilde u^*_\eps$ was defined in \eqref{eq:defutildeeps} satisfying $u^*_\eps=\mathcal{A}^\eps(\tilde u^*_\eps)$. It is therefore reasonable 
to regard \eqref{eq:Mthordereffectiveequation} as a candidate for an effective PDE-constraint. 
The effective energy \eqref{eq:energyeffectiveM} can be justified as follows. Let   $u\in U^*_{ad}$ be a function with compact support $K$ in Fourier space, i.e.,
\begin{equation*}
u(x)=\int_K \hat u(\eta)e^{2\pi i\eta\cdot x}\,d\eta.
\end{equation*}
We readily know that the Bloch coefficients $\hat u_{\eps,m}$ of the adaption $\mathcal{A}^\eps(u)\in U^\eps_{ad}$   satisfy $\hat u_{\eps,m}=0$ for all $m\geq 1$ and $\hat u_{\eps,0}=\hat u$.  Therefore, as a consequence of 
  Lemma \ref{lem:EnergyinBlochform}, we obtain
\begin{align}
\begin{split}
\label{eq:energycompactsupport}
J^\eps(\mathcal A^\eps(u), S_\epsilon \mathcal A^\eps(u) )  =\hat J^\eps(\hat u)=
\frac{\mu_1}{2}\int_{K}\left|\frac{\hat f(\eta) + \hat u(\eta)}{1+\lambda_0^\eps(\eta)}-\hat y_{d,1}(\eta)\right|^2(1+\lambda^\eps_0)\,d\eta\\
 + \frac{\mu_2}{2}\int_{K}\left|\frac{\hat f(\eta) + \hat u(\eta)}{1+\lambda_0^\eps(\eta)}-\hat y_{d,2}(\eta)\right|^2\,d\eta
+\frac{\kappa}{2}\int_{K}|\hat u(\eta)|^2\,d\eta.
\end{split}
\end{align}
In order to find a candidate for the effective cost functional, we proceed analogously to the derivation of the effective PDE-constraint by  
formally replacing all $(1+\lambda_0^\eps(\eta))$-terms in \eqref{eq:energycompactsupport} by $1+P_M^\eps(\eta)$. This leads to the following 
candidate for the effective energy
\begin{align}
\begin{split}
\label{eq:effectiveenergycompactsupport}
\hat J_M^*(\hat u)
=&\frac{\mu_1}{2}\int_{K}\left|\frac{\hat f(\eta) + \hat u(\eta)}{1+P_M^\eps(\eta)}-\hat y_{d,1}(\eta)\right|^2(1+P_M^\eps(\eta))\,d\eta\\
 &+ \frac{\mu_2}{2}\int_{K}\left|\frac{\hat f(\eta) + \hat u(\eta)}{1+P_M^\eps(\eta)}-\hat y_{d,2}(\eta)\right|^2\,d\eta
+\frac{\kappa}{2}\int_{K}|\hat u(\eta)|^2\,d\eta.
\end{split}
\end{align}
By a direct calculation in Fourier space, the above functional is exactly the control-reduced objective functional associated with
 \eqref{eq:energyeffectiveM} taking into account  the effective state equation \eqref{eq:Mthordereffectiveequation}.  
 
\subsection{Well-posedness of \eqref{PM}}
\label{subsec:wellposedness}
The goal of this section is to prove the existence of a unique optimal solution to \eqref{PM},  provided that $\eps>0$ is sufficiently small. 
The main difficulty lies in the well-posedness of the effective equation 
\begin{equation*}
\mathcal{L}^*_M( \tilde y):=\sum_{k=1}^M\eps^{2k-2}(-1)^{k}\mathbb{A}^*_{2k}D^{2k} \tilde y(x) + \tilde y(x)
=f(x)+ u(x),
\end{equation*}  
where the highest-order operator $\mathbb{A}^*_{2M}D^{2M}$ might have no or even the wrong sign. 
Indeed, while $\mathbb{A}^*_{2}$ is positive definite, for $\mathbb{A}^*_{4}$ it has been shown in  \cite{COV2} that this tensor is negative semidefinite. 
In particular, for $M=2$,   \eqref{eq:Mthordereffectiveequation} 
reads as 
\begin{equation*}
-\mathbb{A}^*_2D^2 \tilde y(x)+ \eps^2\mathbb{A}^*_4D^4 \tilde y(x) + \tilde y(x)=f(x)+   u(x).
\end{equation*}
This fourth-order equation is ill-posed for general right hand sides. However, in our setting, $f$ and every control $\tilde u\in U^*_{ad}$ have compact support $K$ in Fourier space, 
wherefore for $\eps>0$ sufficiently small the effective equation can be uniquely solved. The well-posedness is based on the following algebraic observation and 
the fact that the first effective tensor $\mathbb{A}^*_2$ is positive definite, i.e., $\mathbb{A}^*_2\cdot\eta^{\otimes 2}\geq \lambda|\eta|^2$ for some $\lambda>0$.

\begin{lemma}
\label{lem:auxiliarylemma}
Let $K\in\R^n$ be a given compact set and let $M\in\N$ be fixed. Let $\lambda>0$ be the ellipticity constant of $\mathbb{A}^*_2$ as above. 
Then there exists $\eps_{M}>0$ such that for every $0<\eps<\eps_{M}$ and every $N\in\N$ with $N\leq M$ 
\begin{equation}
\label{eq:algebraiclemmaPM}
 P_N^\eps(\eta)\geq \frac{\lambda}{2}|\eta|^2\quad \text{for all } \eta\in K,
\end{equation}
 where $P_N^\eps$ is the polynomial defined in \eqref{eq:polynomialBloch}.
\end{lemma}

\begin{proof}
Since $\mathbb{A}^*_2\cdot\eta^{\otimes 2}\geq \lambda|\eta|^2$, it is sufficient to prove that there exists some $\eps_{M}>0$ such that for every $0<\eps<\eps_{M}$ and every $N\leq M$
\begin{equation} 
\label{eq:algebraiclemmaPMsecondform}
\eps^2(2\pi)^4\mathbb{A}^*_4\cdot\eta^{\otimes 4}+\dots+\eps^{2N-2}(2\pi)^{2N}\mathbb{A}^*_{2N}\cdot\eta^{\otimes 2N}\geq-\frac{\lambda}{2}|\eta|^2
\end{equation} 
Indeed, 
\begin{equation*}
\sum_{k=2}^{N}\left|\eps^{2k-2}(2\pi)^{2k}\mathbb{A}^*_{2k}\cdot\eta^{\otimes 2k}\right|\leq (2\pi)^{2N}|\eta|^2\left(\sum_{k=2}^N \left(\eps|\eta|\right)^{2k-2}\|\mathbb{A}^*_{2k}\|\right),
\end{equation*}
where $\|\mathbb{A}^*_{2k}\|$ denotes an appropriate norm of the tensor $\mathbb{A}^*_{2k}$. 
Since $\eta\in K$ varies in a bounded set, we can choose $\eps_{M}>0$ such that for every $0<\eps<\eps_{M}$
\begin{equation*}
(2\pi)^{2N}\sum_{k=2}^{N}\left(\eps|\eta|\right)^{2k-2}\|\mathbb{A}^*_{2k}\|\leq(2\pi)^{2M}\sum_{k=2}^{M}\left(\eps|\eta|\right)^{2k-2}\|\mathbb{A}^*_{2k}\|<\frac{\lambda}{2}.
\end{equation*}
Then $\eps_{M}$ satisfies the requirements of the lemma. 
\end{proof}

We emphasize that the threshold $\eps_M$ depends on the compact set $K$, the order of approximation $M$ and the effective coefficients $\mathbb{A}^*_{2},\dots,\mathbb{A}^*_{2M}$. 
  With the above auxiliary lemma at hand we can now conclude the well-posedness of the effective Equation \eqref{eq:Mthordereffectiveequation}. 
		
\begin{proposition}[Well-posedness of \eqref{eq:Mthordereffectiveequation}]
\label{prop:wellposednesseffectiveeq} Suppose that Assumption \ref{ass:data} holds.
Let $M\in\N$ and let $\eps_M$ be as in Lemma \ref{lem:auxiliarylemma}. Then, for every $0<\eps<\eps_M$ 
and every control $  u\in L^2(\R^n)$ with Fourier transform $\hat u$ satisfying $\supp(\hat u)\subset K$, it holds that:
\begin{itemize}
\item[i)] There exists a unique weak 
solution $  y\in H^k(\R^n)$ for all $k \in \N$ to the   effective state equation \eqref{eq:Mthordereffectiveequation}. 
\item[ii)] The solution $  y$ satisfies the a priori estimate
\begin{equation}
\label{eq:aprioriestimateeffectiveeq}
\|  y\|^2_{L^2(\R^n)} +\lambda\|\nabla  y\|^2_{L^2(\R^n)} \leq \|f+  u\|^2_{L^2(\R^n)}.
\end{equation}
\end{itemize}
\end{proposition}

\begin{proof}
To show existence, let us consider 
\begin{equation}
\label{eq:representationFouriereffectivesolution}
  y(x):=\int_{K}\frac{\hat f(\eta) + \hat u(\eta)}{1+P_M^\eps(\eta)}e^{2\pi i\eta\cdot x}\,d\eta.
\end{equation} 
As a function with compact support in Fourier space, $y$ belongs to $H^k(\R^n)$ for all $k \in \N$. 
By a direct calculation in Fourier space one finds that \eqref{eq:representationFouriereffectivesolution} is a solution to the effective equation \eqref{eq:Mthordereffectiveequation}. 
It remains to prove the a priori estimate \eqref{eq:aprioriestimateeffectiveeq}, from which uniqueness directly follows. 
Therefore note that the Fourier transform $\hat y$ of every solution satisfies 
\begin{equation*}
(1+P_M^\eps)\hat y=\hat f+\hat u.
\end{equation*}
Testing the above relation with $\hat y$, taking into account $\hat f(\eta)=\hat u(\eta)=\hat y(\eta)=0$ for $\eta\notin K$, and exploiting Lemma \ref{lem:auxiliarylemma} provides that 
\begin{align*}
\int_{\R^n} \left(1+ \frac{\lambda}{2}|\eta|^2\right)|\hat y(\eta)|^2\,d\eta&\leq \int_{K}(1+P_M^\eps(\eta))|\hat y(\eta)|^2\,d\eta
=\int_K (\hat f(\eta)+\hat u(\eta))\overline{\hat y(\eta)}\,d\eta\\
&\leq \|\hat f + \hat u\|_{L^2(K)}\|\hat y\|_{L^2(K)}\leq \frac{1}{2}\left(\|\hat f + \hat u\|^2_{L^2(K)} + \|\hat y\|^2_{L^2(\R^n)}\right),
\end{align*}
where $\bar{\hat y}$ denotes the complex conjugate of $\hat y$.
By Parseval's identity one immediately concludes the desired estimate \eqref{eq:aprioriestimateeffectiveeq}.   
\end{proof}

We have found that the PDE-constraint \eqref{eq:Mthordereffectiveequation} is well-posed. To prove that the corresponding   effective optimal control problem \eqref{PM} has a unique optimal  solution, we reformulate \eqref{PM}
in terms of the Fourier transform. In the same manner the original   problem \eqref{P} 
can be rewritten in terms of the Bloch transform,  
which is particularly important for the error analysis in Section \ref{sec:errorestimates}. 



\begin{proposition}[Reformulation of the optimal control problems]
\label{prop:reformulationoptimalcontrolBlochFourier}
Let Assumptions \ref{ass:data} and  \ref{ass:admissibleset} hold, and let $U^*_{ad}$ be the effective admissible  set from \eqref{eq.effectiveadmissibleset}. Then
\begin{align}
\begin{split}
\label{eq:twoadmsetsequality}
\hat U_{ad}:=&\left\{\hat u\in L^2(K;\C)\,|\, u_{F}\in U^*_{ad}\,\text{with } u_{F}(x)=\int_K\hat u(\eta)e^{2\pi i\eta\cdot x}\,d\eta\right\}\\
=&\left\{\hat u\in L^2(K;\C)\,|\, u_{B}\in (U^\eps_{ad})_K\,\text{with } u_{B}(x)=\int_K\hat u(\eta)\omega_0^\eps(x;\eta)\,d\eta\right\}.
\end{split}
\end{align}
Furthermore:
\begin{itemize}
\item[(i)]    The optimal control problem  \eqref{P} is equivalent to
\begin{equation}\label{Blochoptim}
\begin{aligned}
\min_{\hat u \in \hat U_{ad}}\hat J^\eps(\hat u)
:=\frac{\mu_1}{2}\int_{K}\left|\frac{\hat f(\eta) + \hat u(\eta)}{1+\lambda_0^\eps(\eta)}-\hat y_{d,1}(\eta)\right|^2(1+\lambda^\eps_0(\eta))\,d\eta\\
 + \frac{\mu_2}{2}\int_{K}\left|\frac{\hat f(\eta) + \hat u(\eta)}{1+\lambda_0^\eps(\eta)}-\hat y_{d,2}(\eta)\right|^2\,d\eta
+\frac{\kappa}{2}\int_{K}|\hat u(\eta)|^2\,d\eta
\end{aligned}
\end{equation}
in the sense that $(u^*_\eps,y^*_\eps)$ solves \eqref{P} if and only if the Bloch transform of $u^*_\eps$ solves \eqref{Blochoptim}
\item[(ii)]   If  $0<\eps<\eps_M$ with $\eps_M$  as in Lemma \ref{lem:auxiliarylemma}, then the effective optimal control problem  \eqref{PM}   is equivalent to
\begin{equation}\label{Fourieroptim}
\begin{aligned}
\min_{\hat u \in \hat U_{ad}} \hat J_M^*(\hat u)
:=\frac{\mu_1}{2}\int_{K}\left|\frac{\hat f(\eta) + \hat u(\eta)}{1+P_M^\eps(\eta)}-\hat y_{d,1}(\eta)\right|^2(1+P_M^\eps(\eta))\,d\eta\\
 + \frac{\mu_2}{2}\int_{K}\left|\frac{\hat f(\eta) + \hat u(\eta)}{1+P_M^\eps(\eta)}-\hat y_{d,2}(\eta)\right|^2\,d\eta
+\frac{\kappa}{2}\int_{K}|\hat u(\eta)|^2\,d\eta
\end{aligned}
\end{equation}
in the sense that $(u^*_M,y^*_M)$ solves \eqref{PM} if and only if the Fourier transform of $u^*_M$ solves  \eqref{Fourieroptim}.
\end{itemize}
\end{proposition}
\begin{proof} By virtue of \eqref{eq.effectiveadmissibleset}, \eqref{eq:restrictionK}, and $(U^\eps_{ad})_K \subset U^\eps_{ad}$, we have  $\mathcal{A}^\eps\left(U^*_{ad}\right)=(U^{\eps}_{ad})_K$. 
Furthermore, since for functions with compact support in Fourier space the adaption $\mathcal{A}^\eps$ replaces the Fourier 
basis function $e^{2\pi i\eta\cdot x}$ in the expansion by the Bloch function $\omega_0^\eps(x;\eta)$, it follows that $u_B=\mathcal{A}^\eps(u_F)$.  
In particular the two sets in \eqref{eq:twoadmsetsequality} are equal. 

To prove (i), we  note that by Proposition \ref{prop:representationasadaption} the optimal control $u^*_\eps$ of \eqref{P} lies in the projection 
$(U^\eps_{ad})_K$. This fact along with $\mathcal{A}^\eps\left(U^*_{ad}\right)=(U^{\eps}_{ad})_K$   leads to
$$
\begin{aligned}
u^*_\eps \textrm{ solves }\min_{u \in U^{\eps}_{ad}} J^\eps(u,S_\epsilon u) \quad \Leftrightarrow  \quad u^*_\eps \textrm{ solves } \min_{u \in (U^{\eps}_{ad})_K } J^\eps(u,S_\epsilon u) \\ 
\Leftrightarrow \quad u^*_\eps = \mathcal{A}^\eps(u^*_F) \textrm{ and }  u^*_F \textrm{ solves} \min_{u _F\in U^*_{ad} }  J^\eps(\mathcal A^\eps(u_F),S_\epsilon\mathcal A^\eps(u_F)) \\
\underbrace{\Leftrightarrow}_{ \eqref{eq.effectiveadmissibleset},\, \eqref{eq:energycompactsupport},  \,  \eqref{eq:twoadmsetsequality} } u^*_\eps = \int_K\hat u^*(\eta)\omega_0^\eps(\cdot;\eta)  \,d\eta \textrm{ and }  \hat u^* \textrm{ solves } \min_{\hat u \in \hat U_{ad}}\hat J^\eps(\hat u).
\end{aligned}
$$
The second claim  (ii) follows in an analogous way.   
\end{proof}

    
\begin{remark}
\label{rem:convexcontinuous}
Due to its quadratic form, the effective energy $\hat J^*_M: L^2(K;\C) \to \R$ is convex and   continuous.  
\end{remark}

Finally we prove that the admissible set $\hat U_{ad}$ defined in \eqref{eq:twoadmsetsequality} is nonempty, convex, and  closed. 
Together with Remark \ref{rem:convexcontinuous} this yields the  well-posedness of the effective optimal control problem \eqref{PM} provided $\eps$ is sufficiently small, i.e., 
$\eps<\eps_M$ and   $K\subset Z/\eps$. 

\begin{proposition}
\label{prop:closedhatU}
Let $U^\eps_{ad}$ satisfy Assumption \ref{ass:admissibleset}. Then the set $\hat U_{ad}$ defined as in \eqref {eq:twoadmsetsequality} is nonempty, convex, and closed  in $L^2(K;\C)$. 
\end{proposition}

\begin{proof}
The set $\hat U_{ad}$ is nonempty since for every $u\in U^\eps_{ad} \neq \emptyset$ the projection $u_K$ satisfies $u_K\in \left( U^\eps_{ad}\right)_K$ with the    Bloch coefficient  belonging to $\hat U_{ad}$ according to \eqref{eq:twoadmsetsequality}. 

In view of \eqref{eq:twoadmsetsequality}, the convexity of $\hat U_{ad} \subset L^2(K;\C)$ holds   if $U^*_{ad} \subset L^2(\R^n)$ is convex. By the linearity     of    the adaption $\mathcal{A}^\eps$ and the definition of $U^*_{ad}$  \eqref{eq.effectiveadmissibleset}, the convexity of $U^\eps_{ad} \subset L^2(\R^n)$ implies immediately the convexity of $U^*_{ad} \subset L^2(\R^n)$.

Due to Parseval's identity \eqref{eq:Plancherel}, strong convergence of functions in $L^2(\R^n)$ is equivalent to strong $L^2$-convergence of the corresponding Bloch transforms.   Thus,  it follows from  \eqref{eq:restrictionK}, $(U^\eps_{ad})_K \subset U^\eps_{ad}$, and the closedness of $U^\eps_{ad} \subset L^2(\R^n)$ that  the projected set 
$(U^\eps_{ad})_K$ is    closed in $L^2(\R^n)$, and so by \eqref{eq:twoadmsetsequality} the closedness of $\hat U_{ad} \subset L^2(K;\C)$ follows. \end{proof}

\begin{corollary} Let Assumptions  \ref{ass:data} and   \ref{ass:admissibleset} hold. Furthermore,
let $M\in\N$ and let $\eps_M$ be as in Lemma \ref{lem:auxiliarylemma}. Then, for every $0<\eps<\eps_M$, the effective optimal control problem \eqref{PM} admits a unique optimal solution $(u^*_M, y^*_M)$.
\end{corollary}

To conclude this section we specify the effective admissible sets $\hat U_{ad}$ and $U^*_{ad}$ corresponding to the sets from Example \ref{ex:admissiblesets}. 
At this point, we recall that 
$u\in L^2(\R^d;\C)$ is real valued if and only if $\hat u(-\eta)=\overline{\hat u(\eta)}$. An analogous statement holds for the Bloch transforms.
\noindent
\begin{example}[Effective admissible sets]
Suppose that Assumption \ref{ass:data} holds.  Let $\hat U_{ad}$ be defined through \eqref {eq:twoadmsetsequality} and $U^*_{ad}$ through \eqref{eq.effectiveadmissibleset}.  
\begin{itemize} 
\item[(i)] For $U^\eps_{ad}=L^2(\R^n)$ one has 
\begin{align*}
\hat U_{ad}&=\{L^2(K;\C)\,:\, \hat u(-\eta)=\overline{\hat u(\eta)}\},\\ 
U^*_{ad}&=\left\{u\in L^2(\R^n)\,|\, \supp(\hat u)\subset K\right\}. 
\end{align*}
\item[(ii)] For $U^\eps_{ad}:=\{u\in L^2(\R^n)\,|\,\|u\|_{L^2(\R^n)}\leq L\}$ one has due to Pareval's identity 
\begin{align*}
\hat U_{ad}&=\{\hat u\in L^2(K;\C)\,|\,\hat u(-\eta)=\overline{\hat u(\eta)},\,\|\hat u\|_{L^2(K;\C)}\leq L\},\\
 U^*_{ad}&=\left\{u\in L^2(\R^n)\,|\, \supp(\hat u)\subset K\text{ and } \|u\|_{L^2(\R^n)}\leq L\right\}.
\end{align*}
\end{itemize}
The specification of the admissible sets for $U^\eps_{ad}$ from Example \ref{ex:admissiblesets} in (iii) and (iv) is more subtle and does not allow for a simple form at first sight. 
\begin{itemize}
 \item[(iii)] For $U^\eps_{ad}:=\left\{u\in L^2(\R^n)\,\Big|\,\|S_\epsilon u \|_{L^2(\R^n)}\leq L\right\}$ one obtains, using Lemma \ref{lem:EnergyinBlochform}, 
 \begin{align*}
\hat U_{ad}&=\left\{\hat u\in L^2(K;\C)\,|\,\hat u(-\eta)=\overline{\hat u(\eta)},\, \left\|\tfrac{\hat f+\hat u}{1+\lambda^\eps_0}\right\|_{L^2(K;\C)}\leq L\right\},\\
U^*_{ad}&=\left\{u(x)=\int_K \hat u(\eta)e^{2\pi i\eta\cdot x}\,d\eta\,|\, \hat u\in \hat U_{ad}\right\}.
\end{align*}
 \item[(iv)] For $U^\eps_{ad}:=\left\{u\in L^2(\R^n)\,\Big|E^\eps(S_\epsilon u)\leq L\right\}$ one obtains, using again Lemma \ref{lem:EnergyinBlochform}, 
 \begin{align*}
\hat U_{ad}&=\left\{\hat u\in L^2(K;\C)\,|\,\hat u(-\eta)=\overline{\hat u(\eta)},\,\int_K \frac{|\hat f+\hat u|^2(\eta)}{1+\lambda_0^\eps(\eta)}\,d\eta\leq L\right\},\\
U^*_{ad}&=\left\{u(x)=\int_K \hat u(\eta)e^{2\pi i\eta\cdot x}\,d\eta\,|\, \hat u\in \hat U_{ad}\right\}.
\end{align*}
\end{itemize}
The characterization of the effective set $\hat U_{ad}$ for the above two examples uses the $\eps$-dependent Bloch eigenvalue $\lambda_0^\eps$. To derive a more practicable 
condition one can again replace $\lambda_0^\eps$ by its $M$-th order Taylor expansion $P_M^\eps$. The resulting approximative admissible set $\hat U^M_{ad}$ is close to $\hat U_{ad}$. 
Due to stability of the optimal control problem the minimization can be performed on this approximative set without loosing the approximation quality. 
Recalling that solutions to the effective PDE-constaint \eqref{eq:Mthordereffectiveequation} have the Fourier transform 
$\hat y=\frac{\hat f+\hat u}{1+P_M^\eps}$, it follows  that
\begin{itemize}
 \item[(iii)]
 \begin{align*}
 \hat U^M_{ad}:=&\left\{\hat u\in L^2(K;\C)\,|\,\hat u(-\eta)=\overline{\hat u(\eta)},\, \left\|\tfrac{\hat f+\hat u}{1+P_M^\eps}\right\|_{L^2(K;\C)}\leq L\right\}\\
 =&\left\{\hat u\in L^2(K,\C)\,|\,\hat u(-\eta)=\overline{\hat u(\eta)},\, \|\hat y\|_{L^2(K;\C)}\leq L\right\}.
 \end{align*}
 \item[(iv)] In view of  the definition \eqref{eq:energyeffective equation} of $E^*_M$, denoting by $\mathcal{F}$ the Fourier transform, one obtains
 \begin{align*}
 \hat U^M_{ad}:=&\left\{\hat u\in L^2(K;\C)\,|\,\hat u(-\eta)=\overline{\hat u(\eta)},\,\int_K \frac{|\hat f+\hat u|^2(\eta)}{1+P_M^\eps(\eta)}\,d\eta\leq L\right\}\\
 =&\left\{\hat u\in L^2(K;\C)\,|\,\hat u(-\eta)=\overline{\hat u(\eta)},\,\int_K |\hat y|^2(\eta)(1+P_M^\eps(\eta))\,d\eta\leq L\right\}\\
 =&\left\{\hat u\in L^2(K;\C)\,|\,\hat u(-\eta)=\overline{\hat u(\eta)},\,\hat u=\mathcal{F}(  u)\text{ and } E^*_M(  y)\leq L,\,\text{where }  y\text{ solves }\eqref{eq:Mthordereffectiveequation}\right\}.
 \end{align*}
\end{itemize}
\end{example}

\section{Error estimates}
\label{sec:errorestimates}
Proposition \ref{prop:reformulationoptimalcontrolBlochFourier} allows to rewrite  the original and the effective optimal control problems  in terms of the Fourier/Bloch transforms. 
The fact that both minimization problems \eqref{Blochoptim} and \eqref{Fourieroptim} are constrained by the same admissible set $\hat U_{ad}$ enables us to compare them by 
means of the corresponding variational inequalities. 

In all what follows, we  assume that the order of approximation $M\in\N$ is fixed and that $\eps>0$ is sufficiently small, i.e. $\eps<\eps_M$ and such that $K\subset Z/\eps$.  
We will see that the unique solution $\hat u^*_M \in \hat U_{ad}$ of \eqref{Fourieroptim} is a good approximation  the Bloch transform $\hat u^*_{\eps,0}$ of the optimal control $u^*_\eps$ of $\eqref{P}$, which is exactly the unique solution to \eqref{Blochoptim}. 
This relies on the fact that the cost functionals $\hat J^\eps$ and $\hat J^*_M$ in \eqref{Blochoptim} and \eqref{Fourieroptim} differ only by the factor $1+\lambda_0^\eps$, which in $\hat J^*_M$ is replaced by $1+P^\eps_M$. 
As a consequence, the (unique) optimal control  of \eqref{PM}, which is according to Proposition \ref{prop:reformulationoptimalcontrolBlochFourier} given  by 
\begin{equation}
 \label{eq:effectivecontrolfinalapp}
   u^*_M(x)=\int_K \hat u^*_M(\eta)e^{2\pi i\eta\cdot x}\,d\eta,
\end{equation}
and   the corresponding optimal state $y^*_M$, 
given by the solution to the effective equation
\begin{equation}
\label{eq:effectiveequationfinalapp}
\sum_{k=1}^M\eps^{2k-2}(-1)^{k}\mathbb{A}^*_{2k}D^{2k}  y^*_M(x) +   y^*_M(x)
=f(x)+   u^*_M(x),
\end{equation}
satisfy the desired approximation \eqref{eq:approximationerrorcontrols} with $(\tilde u^*_\epsilon,\tilde y^*_\epsilon)$ as     in   Proposition \ref{prop:representationasadaption}.  Let us now prove our first main result, which states the approximation property of $\hat u^*_M$:

\begin{theorem}[Approximation of the optimal control]
\label{thm:mainestimatecontrols}
Let    $\hat u^*_{\eps,0}\in \hat U_{ad}$ and $\hat u^*_M\in \hat U_{ad}$ denote, respectively, the unique solutions to
  \eqref{Blochoptim} and  \eqref{Fourieroptim}.
Then 
\begin{equation}
\label{eq_estimatecontrolapproximation}
\|\hat u^*_{\eps,0}-\hat u^*_M\|_{L^2(K;\C)}\leq C\eps^{2M}
\end{equation}  
for some $\eps$-independent constant $C>0$.   
\end{theorem}
We emphasize that the constant $C$ in Theorem \ref{thm:mainestimatecontrols} may depend on the order of approximation $M$. 

\begin{proof}[Proof of Theorem \ref{thm:mainestimatecontrols}] As the optimization problems \eqref{Blochoptim} and \eqref{Fourieroptim} are convex with directionally differentiable cost funcitonals,
their solutions $\hat u^*_{\eps,0}$ and $\hat u^*_M$ satisfy the following variational inequalities 
\begin{align}
\label{eq:varinequality1}
\int_K\text{Re}\left[\left(D\hat J^\eps(\hat u^*_{\eps,0})\right)(\eta)(\overline{\hat u}-\overline{\hat u^*_{\eps,0}})(\eta)\right]\,d\eta &\geq 0\quad\text{for all }\hat u\in\hat U_{ad},\\
\label{eq:varinequality2}
\int_K\text{Re}\left[\left(D\hat J^*_{M}(\hat u^*_{M})\right)(\eta)(\overline{\hat u}-\overline{\hat u^*_{M}})(\eta)\right]\,d\eta &\geq 0\quad\text{for all }\hat u\in\hat U_{ad},
\end{align}
where $\text{Re}$ denotes the real part and
\begin{align*}
D\hat J^\eps(\hat u):=&\mu_1\left(\frac{\hat f + \hat u}{1+\lambda_0^\eps}-\hat y_{d,1}\right)
+\mu_2\left(\frac{\hat f + \hat u}{1+\lambda_0^\eps}-\hat y_{d,2}\right)(1+\lambda_0^\eps)^{-1} 
+ \kappa\hat u,\\
D\hat J^*_{M}(\hat u):=&
\mu_1\left(\frac{\hat f + \hat u}{1+P_M^\eps}-\hat y_{d,1}\right)
+\mu_2\left(\frac{\hat f + \hat u}{1+P_M^\eps}-\hat y_{d,2}\right)(1+P_M^\eps)^{-1} 
+ \kappa\hat u.
\end{align*}
Taking $\hat u^*_{\eps,0}$ as a test function in \eqref{eq:varinequality2} and $\hat u^*_{M}$ as test function in \eqref{eq:varinequality1}, we obtain
\begin{align*}
\int_K \text{Re}\left[\left(\mu_1\left(\tfrac{\hat f + \hat u^*_{M}}{1+P_M^\eps}-\hat y_{d,1}\right)
+\mu_2\left(\tfrac{\hat f + \hat u^*_{M}}{1+P_M^\eps}-\hat y_{d,2}\right)(1+P_M^\eps)^{-1} 
+ \kappa\hat u^*_{M}\right)(\overline{\hat u^*_{\eps,0}}-\overline{\hat u^*_{M}})\right](\eta)\,d\eta&\geq 0,\\
\int_K\text{Re}\left[\left(\mu_1\left(\tfrac{\hat f + \hat u^*_{\eps,0}}{1+\lambda_0^\eps}-\hat y_{d,1}\right)
+\mu_2\left(\tfrac{\hat f + \hat u^*_{\eps,0}}{1+\lambda_0^\eps}-\hat y_{d,2}\right)(1+\lambda_0^\eps)^{-1} 
+ \kappa\hat u^*_{\eps,0}\right)(\overline{\hat u^*_{M}}-\overline{\hat u^*_{\eps,0}})\right](\eta)\,d\eta&\geq 0.
\end{align*}
Summing up the two inequalities provides
\begin{align}
\begin{split}
\label{eq:estimateuMustar}
 &\kappa\int_K\left|\hat u^*_{\eps,0}(\eta)-\hat u^*_{M}(\eta)\right|^2(\eta)\,d\eta \\
 \leq \, &\mu_1\int_K\text{Re}
\left[\left(\frac{\hat f + \hat u^*_{M}}{1+P_M^\eps}- 
\frac{\hat f + \hat u^*_{\eps,0}}{1+\lambda_0^\eps}\right)\left(\overline{\hat u^*_{\eps,0}}-\overline{\hat u^*_{M}}\right)\right](\eta)\,d\eta\\
+&\mu_2\int_K\text{Re}\left[
\left(\frac{\hat f + \hat u^*_{M}-(1+P_M^\eps)\hat y_{d,2}}{(1+P_M^\eps)^2}- 
\frac{\hat f + \hat u^*_{\eps,0}-(1+\lambda_0^\eps)\hat y_{d,2}}{(1+\lambda_0^\eps)^2}\right)\left(\overline{\hat u^*_{\eps,0}}-\overline{\hat u^*_{M}}\right)\right](\eta)\,d\eta.
\end{split}
\end{align}
Concerning the first term on the right hand side of \eqref{eq:estimateuMustar} we note that 
\begin{align*}
\frac{\hat f + \hat u^*_{M}}{1+P_M^\eps}- 
\frac{\hat f + \hat u^*_{\eps,0}}{1+\lambda_0^\eps}
=(\hat f + \hat u^*_{\eps,0})\left(\frac{1}{1+P_M^\eps}-\frac{1}{1+\lambda_0^\eps}\right)
+ \frac{1}{1+P^\eps_M}\left(\hat u^*_{M}-\hat u^*_{\eps,0}\right)
\end{align*}
and thus, exploiting that
\begin{equation*}
 \frac{1}{1+P^\eps_M}\left(\hat u^*_{M}-\hat u^*_{\eps,0}\right)\left(\overline{\hat u^*_{\eps,0}}-\overline{\hat u^*_{M}}\right)
=-\frac{1}{1+P^\eps_M}|\hat u^*_{M}-\hat u^*_{\eps,0}|^2\leq 0,
\end{equation*} 
it follows that
\begin{align*}
&\mu_1\int_K\text{Re}
\left[\left(\frac{\hat f + \hat u^*_{M}}{1+P_M^\eps}- 
\frac{\hat f + \hat u^*_{\eps,0}}{1+\lambda_0^\eps}\right)\left(\overline{\hat u^*_{\eps,0}}-\overline{\hat u^*_{M}}\right)\right](\eta)\,d\eta\\
\leq&\,\mu_1 \left(\sup_{\eta\in K}\left|\frac{1}{1+P_M^\eps(\eta)}-\frac{1}{1+\lambda_0^\eps(\eta)}\right|\right)
\|\hat u^*_{\eps,0}-\hat u^*_{M}\|_{L^2(K;\C)}\|\hat f + \hat u^*_{\eps,0}\|_{L^2(K;\C)}\\
\leq& C\eps^{2M}\,\|\hat u^*_{\eps,0}-\hat u^*_{M}\|_{L^2(K;\C)}\|\hat f + \hat u^*_{\eps,0}\|_{L^2(K;\C)},
\end{align*}
where in the last step we have used that $|P_M^\eps(\eta)-\lambda_0^\eps(\eta)|\leq C\eps^{2M}$ uniformly in $\eta\in K$. 
For the the second term on the right hand side of \eqref{eq:estimateuMustar} we argue analogously,
\begin{align*}
&\mu_2\int_K\text{Re}\left[
\left(\frac{\hat f + \hat u^*_{M}-(1+P_M^\eps)\hat y_{d,2}}{(1+P_M^\eps)^2}- 
\frac{\hat f + \hat u^*_{\eps,0}-(1+\lambda_0^\eps)\hat y_{d,2}}{(1+\lambda_0^\eps)^2}\right)\left(\overline{\hat u^*_{\eps,0}}-\overline{\hat u^*_{M}}\right)\right](\eta)\,d\eta\\
\leq \, &\mu_2 \left(\sup_{\eta\in K}\left|\frac{1}{(1+P_M^\eps(\eta))^2}-\frac{1}{(1+\lambda_0^\eps(\eta))^2}\right|\right)
\|\hat u^*_{\eps,0}-\hat u^*_{M}\|_{L^2(K)}\|\hat f + \hat u^*_{\eps,0}\|_{L^2(K;\C)}\\
+&\mu_2 \left(\sup_{\eta\in K}\left|\frac{1}{1+\lambda_0^\eps(\eta)}-\frac{1}{1+P^\eps_M(\eta)}\right|\right)
\|\hat u^*_{\eps,0}-\hat u^*_{M}\|_{L^2(K;\C)}\|\hat y_{d,2}\|_{L^2(K;\C)}\\
\leq& C\eps^{2M}\,\|\hat u^*_{\eps,0}-\hat u^*_{M}\|_{L^2(K;\C)}\left(\|\hat f + \hat u^*_{\eps,0}\|_{L^2(K;\C)} + \|\hat y_{d,2}\|_{L^2(K;\C)}\right). 
\end{align*}
Summing up, 
\begin{equation*}
\kappa\|\hat u^*_{\eps,0}-\hat u^*_{M}\|^2_{L^2(K;\C)}
 \leq C\eps^{2M}\left(2\|\hat f + \hat u^*_{\eps,0}\|_{L^2(K;\C)} + \|\hat y_{d,2}\|_{L^2(K;\C)}\right)\|\hat u^*_{\eps,0}-\hat u^*_{M}\|_{L^2(K;\C)}.
\end{equation*}
This is the claimed result \eqref{eq_estimatecontrolapproximation}, since by Parseval's identity 
$\| \hat u^*_{\eps,0}\|_{L^2(K;\C)}=\|u^*_\eps\|_{L^2(\R^n)}$, and $\|u^*_\eps\|_{L^2(\R^n)}$ is uniformly bounded (Lemma  \ref{lem:assumption2}). 
\end{proof}
Exploiting Parseval's identity along with \eqref{eq:defutildeeps} and \eqref{eq:effectivecontrolfinalapp}, a straightforward consequence of Theorem \ref{thm:mainestimatecontrols}    is the following approximation property for the optimal control of \eqref{PM}:
\begin{corollary}
\label{cor:errorcontrols} Let $\tilde u^*_\eps$  be  as in \eqref{eq:defutildeeps}. Then 
the unique optimal control $u^*_M$ to \eqref{PM}  satisfies 
\begin{equation}
\|\tilde u^*_\eps-  u^*_M\|_{L^2(\R^n)}\leq C\eps^{2M}
\end{equation}
for some $\eps$-independent constant $C>0$. 
\end{corollary}

We next show that the solution $  y^*_M$ of \eqref{eq:effectiveequationfinalapp} is the desired approximation of $\tilde y^*_\eps$. 

\begin{proposition}[Approximation of the solution to the PDE constraint]
\label{prop:approximationPDEconstraint}
Let $  y^*_M$ denote the solution to the effective equation \eqref{eq:effectiveequationfinalapp} and let $\tilde y^*_\eps$ be as in  \eqref{eq:defytildeeps}. 
Then 
\begin{equation}
\|\tilde y^*_\eps-  y^*_M\|_{L^2(\R^n)}\leq C\eps^{2M}
\end{equation}
for an $\eps$-independent constant $C$. 
\end{proposition}
\begin{proof}
In Proposition \ref{prop:firstapproximationy} we have already shown that the following function
\begin{equation*}
\label{eq:ystarMapp}
\tilde y^*_{M,app}(x) =\int_{K}\frac{\hat f(\eta) + \hat u^*_{\eps,0}(\eta)}{1+P_M^\eps(\eta)}e^{2\pi i\eta\cdot x}\,d\eta.
\end{equation*} 
satisfies $\|\tilde y^*_\eps- \tilde y^*_{M,app}\|_{L^2(\R^n)}\leq C\eps^{2M}$. By triangle inequality it is sufficient to prove
\begin{equation*}
\|\tilde y^*_{M,app}-  y^*_M\|_{L^2(\R^n)}\leq C\eps^{2M}.
\end{equation*}
Indeed, $\tilde y^*_{M,app}$ solves the effective equation with right hand side $f+\tilde u_\eps^*$, while $  y^*_M$ solves the same equation with right hand side $f+  u^*_M$.  
By Corollary \ref{cor:errorcontrols} one has  $\|(f+\tilde u_\eps^*)-(f+  u^*_M)\|_{L^2(\R^n)}=\|\tilde u_\eps^*-\tilde u^*_M\|_{L^2(\R^n)}\leq C\eps^{2M}$. Exploiting the linearity of 
the effective equation and the a priori estimate \eqref{eq:aprioriestimateeffectiveeq} from Proposition \ref{prop:wellposednesseffectiveeq}, one directly concludes
\begin{equation*}
\|\tilde y^*_{M,app}-\tilde y^*_M\|_{L^2(\R^n)}\leq \|\tilde u_\eps^*-  u^*_M\|_{L^2(\R^n)}\leq C\eps^{2M},
\end{equation*}
which was the claim.
\end{proof}
We conclude this section by discussing high-order approximations of the adjoint state corresponding to the optimal control problem \eqref{P}. 
By standard arguments (cf. \cite{T}), the adjoint state $p^*_\eps$ is characterized through the following adjoint equation:
\begin{equation*}
 -\nabla\cdot\left(A^\eps \nabla p^*_\eps\right) + p^*_\eps=\mu_1 \left(-\nabla\cdot\left(A^\eps\nabla (y^*_\eps-y^\eps_{d,1}\right) +(y^*_\eps-y^\eps_{d,1})\right) 
 +\mu_2(y^*_\eps-y^\eps_{d,2}).
\end{equation*}
It can therefore be written as a sum 
\begin{equation}
\label{eq:pepssum}
 p^*_\eps=\mu_1(y^*_\eps-y^\eps_{d,1})+ \mu_2 p^*_{\eps,1},\end{equation}
where $p_{\eps,1}$ solves 
\begin{equation} \label{eq:p1}
-\nabla\cdot\left(A^\eps \nabla p^*_{\eps,1}\right) + p^*_{\eps,1}= y^*_\eps-y^\eps_{d,2}.
\end{equation}
By Proposition \ref{prop:representationasadaption} and \eqref{fudwellprepared}, the right hand sides of \eqref{eq:pepssum} and \eqref{eq:p1} are nothing but
$\mu_1\mathcal{A^\eps}(\tilde y^*_\eps-y_{d,1}) + \mu_2 p_{\eps,1}^*$ and $ \mathcal{A}^\eps(\tilde y^*_\eps-y_{d,2})$.
Therefore, analogously to the approximation results for $u_\eps^*$ and $y_\eps^*$, one can derive that 
\begin{equation*}
 \|p_{\eps,1}^*-\mathcal{A^\eps}(  p^*_{M})\|_{L^2(\R^n)}\leq C\eps^{2M}
\end{equation*}
for $  p^*_{M}$ solving the $M$-th order approximate equation 
\begin{equation}
 \sum_{k=1}^M\eps^{2k-2}(-1)^{k}\mathbb{A}^*_{2k}D^{2k}  p^*_{M} +   p^*_{M}
=  y^*_M-y_{d,2}.
\end{equation}
Summing up, the following $M$-th order approximation result is 
obtained for the adjoint state:
\begin{equation*}
 \left\|p_\eps^*-\left(\mu_1\mathcal{A}^\eps(  y^*_M-y_{d,1})+\mu_2\mathcal{A}^\eps(  p^*_M)\right)\right\|_{L^2(\R^n)}\leq C\eps^{2M},
\end{equation*}
whose proof is skipped for the sake of brevity.

\section{A well-posed effective problem}
\label{sec:Wellposedepsindependent}
The effective optimal control problem from Definition \ref{def:highordereffectiveproblem} is well-posed only for $\eps$ lying below the threshold $\eps_M$ introduced in  
Lemma \ref{lem:auxiliarylemma}. In this section we provide an alternative effective problem which does not require the smallness of $\eps$. To keep the presentation simple and to 
demonstrate the main principle we restrict ourselves to second order approximations, i.e. $M=2$. However by an induction argument the analysis can 
be extended to higher-order approximations.

The central idea relies on the following formal calculation. For $M=2$ the effective PDE constraint \eqref{eq:Mthordereffectiveequation} in one space dimension reads as 
\begin{equation}
\label{eq:equationM21D}
-\mathbb{A}^*_{2}\partial^2_{x}y(x)+\eps^2\mathbb{A}^*_{4}\partial^4_x y(x)+ y(x) =f(x)+ u(x),
\end{equation}    
with $\mathbb{A}^*_{2}> 0$ and $\mathbb{A}^*_{4}\leq 0$ as shown in \cite{COV2}. We thus formally have $\mathbb{A}^*_{2}\partial^2_x  y= y-f-  u+O(\eps^2)$. By 
rewriting $\eps^2\mathbb{A}^*_{4}\partial^4_x y=\eps^2\frac{\mathbb{A}^*_{4}}{\mathbb{A}^*_{2}}\partial^2_x(\mathbb{A}^*_{2}\partial^2_x y)$ and formally replacing 
$\mathbb{A}^*_{2}\partial^2_x y$ in \eqref{eq:equationM21D} by $ y-f-  u$ we obtain the following effective equation
\begin{equation*}
-\left(\mathbb{A}^*_{2}-\eps^2\frac{\mathbb{A}^*_{4}}{\mathbb{A}^*_{2}}\right)\partial^2_{x} y(x)+ y(x) 
=\left(\text{Id}+\eps^2\frac{\mathbb{A}^*_{4}}{\mathbb{A}^*_{2}} \partial_x^2\right)(f(x)+ u(x)),
\end{equation*}  
which is well-posed, since $\mathbb{A}^*_{2}-\eps^2\frac{\mathbb{A}^*_{4}}{\mathbb{A}^*_{2}}>0$. In this simple form the above replacement procedure is known as the Boussinesq trick. 
In higher space dimension a similar, but more complicated calculation can be done. It has been shown in \cite{ABV} and \cite{DLS} that every tensor 
$\mathbb{A}^*_{4}\in\R^{n\times n\times n\times n}$ allows for the following decomposition 
\begin{equation}
\label{eq:decomposition}
\mathbb{A}^*_4D^4=-\mathbb{B}^*_2D^2(\mathbb{A}^*_2D^2) + \mathbb{B}^*_4D^4
\end{equation}
with symmetric, positive semidefinite tensors $\mathbb{B}^*_2\in\R^{n\times n}$ and $\mathbb{B}^*_4\in\R^{n\times n\times n\times n}$. 
Proceeding exactly as in the one dimensional case we end up with the following effective PDE constraint
\begin{equation}
\label{eq:effectivePDEconstraintwellposed}
-(\mathbb{A}^*_2+ \eps^2\mathbb{B}^*_2)D^2 y(x) + \eps^2\mathbb{B}^*_4D^4 y(x) +  y(x)
=\left(\text{Id}-\eps^2\mathbb{B}^*_2D^2\right)(f(x) +  u(x)),
\end{equation}
which is well-posed independently of the choice of $\eps$. In the recent publication \cite{AP3} the Boussinesq trick has been generalized to arbitrary order $M>2$ 
using an induction argument.  

\begin{proposition}[Well posedness of the PDE constraint]
\label{prop:Apriorisolutionnew}
Let $u\in L^2(\R^n)$ be a given control with Fourier transform $\hat u$ and $\supp(\hat u)\subset K$. 
\begin{itemize}
\item[i)] There exists a unique weak solution $y\in H^1(\R^n)$ to Equation \eqref{eq:effectivePDEconstraintwellposed}.
\item[ii)] The solution $y$ satisfies the a priori estimate 
\begin{equation}
\label{eq:apriorinewwffectivePDE}
\| y\|^2_{L^2(\R^n)} + 2\lambda\|\nabla y\|^2_{L^2(\R^n)}\leq C^2\|f+u\|^2_{L^2(\R^n)}
\end{equation}
with a constant $C$ depending only on $K$. Here $\lambda>0$ is the ellipticity constant of $\mathbb{A}^*_2$, i.e. 
such that $\mathbb{A}^*_2\cdot\eta^{\otimes 2}\geq \lambda|\eta|^2$ for every $\eta\in\R^n$. 
\end{itemize}  
\end{proposition}

\begin{proof}
Consider
\begin{align*}
Q^\eps_2(\eta)&:=\eps^2(2\pi)^2\mathbb{B}^*_2\cdot\eta^{\otimes 2},\\
R^\eps_2(\eta)&:=(2\pi)^2(\mathbb{A}^*_2+ \eps^2\mathbb{B}^*_2)\cdot\eta^{\otimes 2} 
+ \eps^2(2\pi)^4\mathbb{B}^*_4\cdot\eta^{\otimes 4}.
\end{align*}
By a direct calculation in Fourier space one finds that 
\begin{equation}
\label{eq:formulasolutioneffectivePDEnew}
y(x):=\int_K \frac{1+Q^\eps_2(\eta)}{1+R^\eps_2(\eta)}(\hat f(\eta) + \hat u(\eta))e^{2\pi i\eta\cdot x}\,d\eta
\end{equation}
is a solution to \eqref{eq:effectivePDEconstraintwellposed}. Moreover, every solution $\tilde y$ with Fourier transform $\hat y$ satisfies
\begin{equation*}
(1+R^\eps_2(\eta))\hat y(\eta)=(1+Q^\eps_2(\eta))(\hat f(\eta) + \hat u(\eta)).
\end{equation*}
Multiplying the above relation with $\hat y$, integrating over $\R^n$, taking into account $\hat f(\eta)=\hat u(\eta)=\hat y(\eta)=0$ for $\eta\notin K$, and exploiting that $\mathbb{B}^*_2$ and $\mathbb{B^*}_4$ 
are positive semidefinite one finds
\begin{align*}
&\int_{\R^n}\left(1+\lambda|\eta|^2\right)|\hat y(\eta)|^2\,d\eta\leq\int_K(1+R^\eps_2(\eta))|\hat y(\eta)|^2\,d\eta\\
=&\int_K(1+Q^\eps_2(\eta))(\hat f(\eta) + \hat u(\eta))\bar{\hat y}(\eta)\,d\eta
\leq C\|\hat f+\hat u\|_{L^2(K;\C)}\|\hat y\|_{L^2(K;\C)}\\
\leq&\frac12\left(C^2\|\hat f+\hat u\|^2_{L^2(K;\C)} + \|\hat y\|^2_{L^2(K;\C)}\right). 
\end{align*}
The desired estimate \eqref{eq:apriorinewwffectivePDE} follows directly by Parseval's identity.  
\end{proof}

It is interesting to note how the prefactor $\frac{1+Q^\eps_2}{1+R^\eps_2}$ in \eqref{eq:formulasolutioneffectivePDEnew} is related to the former prefactors $\frac{1}{1+P^\eps_2}$ and $\frac{1}{1+\lambda_0^\eps}$, 
see \eqref{eq:representationFouriereffectivesolution} and \eqref{eq:energycompactsupport}. Actually
\begin{align*}
&(1+Q^\eps_2(\eta))(1+P_2^\eps(\eta))-(1+R^\eps_2(\eta))\\
=&
(1+\eps^2(2\pi)^2\mathbb{B}^*_2\cdot\eta^{\otimes 2})(1+(2\pi)^2\mathbb{A}^*_2\cdot\eta^{\otimes 2} + \eps^2(2\pi)^4\mathbb{A}^*_4\cdot\eta^{\otimes 4})\\
-&\left(1+(2\pi)^2(\mathbb{A}^*_2+ \eps^2\mathbb{B}^*_2)\cdot\eta^{\otimes 2} 
+ \eps^2(2\pi)^4\mathbb{B}^*_4\cdot\eta^{\otimes 4}\right) + O(\eps^4)\\
=&\,\eps^2(2\pi)^4\left(\mathbb{B}^*_2\cdot\eta^{\otimes 2}\mathbb{A}^*_2\cdot\eta^{\otimes 2} + \mathbb{A}^*_4\cdot\eta^{\otimes 4}-\mathbb{B}^*_4\cdot\eta^{\otimes 4}\right) + O(\eps^4)
\stackrel{\eqref{eq:decomposition}}{=}O(\eps^4),
\end{align*}
 where the error is of order $\eps^4$ uniformly in $\eta\in K$. In particular exploiting the fact that $R^\eps_2$ and $\lambda_0^\eps$ are nonnegative, we conclude that
\begin{align}
\begin{split}
\label{eq:neqfactorclosetoold}
\left|\frac{1+Q^\eps_2(\eta)}{1+R^\eps_2(\eta)}-\frac{1}{1+\lambda_0^\eps(\eta)}\right|
=&\left|\frac{(1+Q^\eps_2(\eta))(1+\lambda_0^\eps(\eta))-(1+R^\eps_2(\eta)}{(1+R^\eps_2(\eta))(1+\lambda_0^\eps(\eta))}\right|\\
=&\left|\frac{(1+Q^\eps_2(\eta))(1+P^\eps_2(\eta) + O(\eps^4))-(1+R^\eps_2(\eta))}{(1+R^\eps_2(\eta))(1+\lambda_0^\eps(\eta))}\right|\\
\leq&|(1+Q^\eps_2(\eta))(1+P^\eps_2(\eta))-(1+R^\eps_2(\eta))| + O(\eps^4)\\
=&O(\eps^4)
\end{split}
\end{align}
uniformly in $\eta\in K$. 
We emphasize once more that the advantage of working with $\frac{1+Q^\eps_2}{1+R^\eps_2}$ instead of $\frac{1}{1+P^\eps_2}$ is that $1+R^\eps_2\geq 1>0$ independently of $\eps$, 
while this is not the case for $1+P^\eps_2$. In view of \eqref{eq:neqfactorclosetoold}, it is reasonable to introduce the following effective cost functional in Fourier form to obtain a second order 
approximation of the original cost functional $\hat J^\eps$ from \eqref{eq:energycompactsupport}. 

\begin{definition}
[Well posed effective optimization problem in Fourier form]
\label{def:alternativeeffectivemodel}
Let $\hat U_{ad}$ be the effective admissible set from \eqref{eq:twoadmsetsequality} and let $\mathbb{B}^*_2, \mathbb{B}^*_4$ be as in \eqref{eq:decomposition}. 
Minimize the cost functional 
\begin{align}
\begin{split}
\hat I_2^*(\hat u)
=&\frac{\mu_1}{2}\int_{K}\left|\frac{1+Q^\eps_2(\eta)}{1+R^\eps_2(\eta)}(\hat f(\eta) + \hat u(\eta))-\hat y_{d,1}(\eta)\right|^2\frac{1+R^\eps_2(\eta)}{1+Q^\eps_2(\eta)}\,d\eta\\
 &+ \frac{\mu_2}{2}\int_{K}\left|\frac{1+Q^\eps_2(\eta)}{1+R^\eps_2(\eta)}(\hat f(\eta) + \hat u(\eta))-\hat y_{d,2}(\eta)\right|^2\,d\eta
+\frac{\kappa}{2}\int_{K}|\hat u(\eta)|^2\,d\eta
\end{split}
\end{align}
in the set $\hat U_{ad}$. 
\end{definition}  
Let us comment on the above optimization problem. The problem is well-posed independently of $\eps$ with a unique minimizer $\hat u^{*}_{2,\hat I}\in \hat U_{ad}$.  
Analogously to Theorem \ref{thm:mainestimatecontrols} we prove in Theorem \ref{thm:approxoptimacontrolnew} below that 
\begin{equation*}
\|\hat u^*_{\eps,0}-\hat u^*_{2,\hat I}\|_{L^2(K;\C)}\leq C\eps^{4}.
\end{equation*} 
Exactly as in Proposition \ref{prop:approximationPDEconstraint} we then conclude, using the a priori estimate from Proposition \ref{prop:Apriorisolutionnew}, that 
\begin{equation}
 y^*_{2,\hat I}(x):=\int_K\frac{1+Q^\eps_2(\eta)}{1+R^\eps_2(\eta)}(\hat f(\eta) + \hat u^{*}_{2,\hat I}(\eta))e^{2\pi i\eta\cdot x}\,d\eta
\end{equation} 
satisfies the error estimate 
\begin{equation}
\| y^*_\eps-y^*_{2,\hat I}\|_{L^2(\R^n)}\leq C\eps^4
\end{equation}
for some $\eps$-independent constant $C$. 
The optimization problem from Definition \ref{def:alternativeeffectivemodel} is thus an alternative to approximate the original problem up to errors of order $\eps^4$. 
Its central advantage lies in the fact that it is well-posed independently of $\eps$. 

\begin{remark}[higher-order well-posed approximations] $ $
\begin{itemize}
\item[i)]
 As shown in \cite{AP3} the Boussinesq trick can be generalized to arbitrary order $M\in\N$. This fact can be used to derive a higher-order effective optimal control problem that is well-posed 
 independently of $\eps$. 
 \item[ii)] In \cite{BG} a different regualarization approach of the effective equation \eqref{eq:Mthordereffectiveequation} is proposed which relies on adding 
 a high-order spatial derivative. \cite{AP3} provides a brief comparision of both regularization methods.    
 \end{itemize}
\end{remark}
We now prove the approximation property of the alternative effective optimization problem from Definition \ref{def:alternativeeffectivemodel}.
\begin{theorem}[Approximation of the optimal control]
\label{thm:approxoptimacontrolnew}
Let $\hat u^{*}_{2,\hat I}\in \hat U_{ad}$ be the unique minimizer of $\hat I_2^*$ and let $\hat u^*_{\eps,0}\in \hat U_{ad}$ be the unique minimizer of $\hat J^\eps$. 
Then 
\begin{equation}
\label{eq_estimatecontrolapproximation}
\|\hat u^*_{\eps,0}-\hat u^{*}_{2,\hat I}\|_{L^2(K;\C)}\leq C\eps^{4}.
\end{equation}  
for an $\eps$-independent constant $C$.    
\end{theorem}

\begin{proof}
The proof follows the lines of Theorem \ref{thm:mainestimatecontrols}. With
\begin{equation*}
D\hat I_2^*(\hat u):=\mu_1\left(\frac{1+Q^\eps_2}{1+R^\eps_2}(\hat f + \hat u)-\hat y_{d,1}\right)
+\mu_2\left(\frac{1+Q^\eps_2}{1+R^\eps_2}(\hat f + \hat u)-\hat y_{d,2}\right)\frac{1+Q^\eps_2}{1+R^\eps_2} + \kappa\hat u,
\end{equation*}
taking $\hat u^*_{\eps,0}$ as a test function in the variational inequality corresponding to $\hat I_2^*$ and $\hat u^{*}_{2,\hat I}$ as a test function in \eqref{eq:varinequality1}, one obtains
\begin{align*}
\int_K \text{Re}\bigg[\bigg(&\mu_1\left(\tfrac{1+Q^\eps_2}{1+R^\eps_2}(\hat f + \hat u^{*}_{2,\hat I})-\hat y_{d,1}\right)
+\mu_2\left(\tfrac{1+Q^\eps_2}{1+R^\eps_2}(\hat f + \hat u^{*}_{2,\hat I})-\hat y_{d,2}\right)\tfrac{1+Q^\eps_2}{1+R^\eps_2}+\kappa\hat u^{*}_{2,\hat I}\bigg)\\
&\times(\overline{\hat u^*_{\eps,0}}-\overline{\hat u^{*}_{2,\hat I}})\bigg](\eta)\,d\eta\geq 0
\end{align*}
and
\begin{equation*}
\int_K\text{Re}\left[\left(\mu_1\left(\tfrac{\hat f + \hat u^*_{\eps,0}}{1+\lambda_0^\eps}-\hat y_{d,1}\right)
+\mu_2\left(\tfrac{\hat f + \hat u^*_{\eps,0}}{1+\lambda_0^\eps}-\hat y_{d,2}\right)(1+\lambda_0^\eps)^{-1} 
+ \kappa\hat u^*_{\eps,0}\right)(\overline{\hat u^{*}_{2,\hat I}}-\overline{\hat u^*_{\eps,0}})\right](\eta)\,d\eta\geq 0.
\end{equation*}
Summing up the two inequalities one concludes
\begin{align*}
 &\kappa\int_K\left|\hat u^*_{\eps,0}(\eta)-\hat u^{*}_{2,\hat I}(\eta)\right|^2(\eta)\,d\eta \\
 \leq \, &\mu_1\int_K\text{Re}
\left[\left(\tfrac{1+Q^\eps_2}{1+R^\eps_2}(\hat f + \hat u^{*}_{2,\hat I})- 
\tfrac{1}{1+\lambda_0^\eps}(\hat f + \hat u^*_{\eps,0})\right)\left(\overline{\hat u^*_{\eps,0}}-\overline{\hat u^{*}_{2,\hat I}}\right)\right](\eta)\,d\eta\\
+&\mu_2\int_K\text{Re}
\left[\left(\left(\tfrac{1+Q^\eps_2}{1+R^\eps_2}(\hat f + \hat u^{*}_{2,\hat I})-\hat y_{d,2}\right)\tfrac{1+Q^\eps_2}{1+R^\eps_2}- 
\left(\tfrac{\hat f + \hat u^*_{\eps,0}}{1+\lambda_0^\eps}-\hat y_{d,2}\right)(1+\lambda_0^\eps)^{-1} \right)\left(\hat u^*_{\eps,0}-\hat u^{*}_{2,\hat I}\right)\right](\eta)\,d\eta.
\end{align*}
Since by \eqref{eq:neqfactorclosetoold} the difference between $\frac{1+Q^\eps_2}{1+R^\eps_2}$ and $\frac{1}{1+\lambda_0^\eps}$ is of order $O(\eps^4)$ uniformly in $\eta\in K$, we 
can argue exactly as in the proof of Theorem  \ref{thm:mainestimatecontrols} to find 
\begin{equation*}
\kappa\|\hat u^*_{\eps,0}(\eta)-\hat u^{*}_{2,\hat I}\|^2_{L^2(K;\C)}\leq C\eps^4\|\hat u^*_{\eps,0}(\eta)-\hat u^{*}_{2,\hat I}\|_{L^2(K;\C)},
\end{equation*}
which is the claimed result. 
\end{proof}
We conclude this section by rewriting the minimization problem from Definition \ref{def:alternativeeffectivemodel} as an optimal control 
problem with PDE constraint \eqref{eq:effectivePDEconstraintwellposed}. Unfortunately, the factor $\frac{1+R^\eps_2(\eta)}{1+Q^\eps_2(\eta)}$ in the first part of 
the cost functional $\hat I^*_2$ does 
not allow for a simple representation of the cost functional in the corresponding optimal control problem. 

\begin{proposition}[Reformulation as an optimal control problem]
\label{prop:reformulationnewproblem}
The minimization problem from Definition \ref{def:alternativeeffectivemodel} is equivalent to minimizing the cost functional
\begin{align*}
I^*_2( u, y, w):=&
\frac{\mu_1}{2}\int_{\R^n}
\Big(\left\langle D( w-w_{d,1})(x), D( w-w_{d,1})(x)
\right\rangle_{\mathbb{A}^*_{2} + 2\eps^2\mathbb{B}^*_2}\\
&\quad+\eps^2\left\langle D^2( w-w_{d,1})(x), D^2( w-w_{d,1})(x)
\right\rangle_{\mathbb{B}^*_4 + \mathbb{A}^*_2\mathbb{B}^*_2 + \eps^2\mathbb{B}^*_2\mathbb{B}^*_2}\\
&\quad+\eps^4\left\langle D^3( w-w_{d,1})(x), D^3( w-w_{d,1})(x)
\right\rangle_{\mathbb{B}^*_4\mathbb{B}^*_2}
 + | w(x)-w_{d,1}(x)|^2\Big)\,dx\\
+&\frac{\mu_2}{2}\int_{\R^n}| y(x)-y_{d,2}(x)|^2\,dx + \frac{\kappa}{2}\int_{\R^n}| u(x)|^2\,dx
\end{align*} 
in the set $U^*_{ad}$ under the PDE constraint \eqref{eq:effectivePDEconstraintwellposed}. The functions $ w, w_{d,2}$ are determined through
\begin{equation*}
-\eps^2\mathbb{B}^*_2D^2 w(x) +  w(x)= y(x)\quad\text{and}\quad -\eps^2\mathbb{B}^*_2D^2 w_{d,2}(x) + w_{d,2}(x)=y_{d,2}(x),
\end{equation*}
where the composition of tensors is defined as 
\begin{align*}
\left(\mathbb{A}^*_2\mathbb{B}^*_2\right)_{i_1,i_2,i_3,i_4}&:=(\mathbb{A}^*_2)_{i_1,i_3}(\mathbb{B}^*_2)_{i_2,i_4}\quad\text{and}\quad
\left(\mathbb{B}^*_2\mathbb{B}^*_2\right)_{i_1,i_2,i_3,i_4}:=(\mathbb{B}^*_2)_{i_1,i_3}(\mathbb{B}^*_2)_{i_2,i_4},\\
\left(\mathbb{B}^*_4\mathbb{B}^*_2\right)_{i_1,\dots,i_6}&:=(\mathbb{B}^*_4)_{i_1,i_2,i_4,i_5}(\mathbb{B}^*_2)_{i_3,i_6}.
\end{align*}
In particular $\hat u^{*}_{2,\hat I}$ is the Fourier transform of the unique optimal control $ u^{*}_{2,I}$. 
\end{proposition}
Note that the above definition is designed in such way that the composition of two symmetric tensors is again symmetric. Hereby symmetry is understood in weak sense. We say 
that a 4-tensor $\mathbb{A}\in\R^{n^4}$ and a 6-tensor $\mathbb{B}\in\R^{n^6}$ are symmetric, if $(\mathbb{A})_{i_1,i_2,i_3,i_4}=(\mathbb{A})_{i_3,i_4,i_1,i_2}$ and 
$(\mathbb{B})_{i_1,i_2,i_3,i_4,i_5,i_6}=(\mathbb{B})_{i_4,i_5,i_6,i_1,i_2,i_3}$ for all indeces $i_1,\dots,i_6\in\{1,\dots,n\}$. 

\begin{proof}[Proof of Proposition \ref{prop:reformulationnewproblem}]
The second and third part of the energy are clear due to Parseval's identity. It remains to justify the first part of $I^*_2$. 
Indeed, denoting the Fourier transform of $\tilde y$ by $\hat y,$ the first part on the cost functional 
$\hat I^*_2$ from  Definition \ref{def:alternativeeffectivemodel} reads as 
\begin{align*}
&\frac{\mu_1}{2}\int_{K}\left|\frac{1+Q^\eps_2(\eta)}{1+R^\eps_2(\eta)}(\hat f(\eta) + \hat u(\eta))-\hat y_{d,1}(\eta)\right|^2\frac{1+R^\eps_2(\eta)}{1+Q^\eps_2(\eta)}\,d\eta\\
=\,&\frac{\mu_1}{2}\int_{K}\left|\hat y(\eta)-\hat y_{d,1}(\eta)\right|^2\frac{1+R^\eps_2(\eta)}{1+Q^\eps_2(\eta)}\,d\eta\\
=\,&\frac{\mu_1}{2}\int_{K}\left|\frac{\hat y(\eta)-\hat y_{d,1}(\eta)}{1+Q^\eps_2(\eta)}\right|^2(1+R^\eps_2(\eta))(1+Q^\eps_2(\eta))\,d\eta.
\end{align*}
By a direct calculation in Fourier space one finds that $\frac{\hat y(\eta)}{1+Q^\eps_2(\eta)}$ is the Fourier transform of $ w$ and $\frac{\hat y_{d,1}(\eta)}{1+Q^\eps_2(\eta)}$ the 
Fourier transform of $w_{d,2}$. Moreover, 
\begin{align*}
&(1+R^\eps_2(\eta))(1+Q^\eps_2(\eta))\\
=&1+(2\pi)^2\left[\mathbb{A}^*_2 + 2\eps^2\mathbb{B}^*_2\right]\cdot\eta^{\otimes 2} + 
\eps^2(2\pi)^4\left[\mathbb{B}^*_4 + \mathbb{A}^*_2\mathbb{B}^*_2 + \eps^2\mathbb{B}^*_2\mathbb{B}^*_2\right]\cdot\eta^{\otimes 4}
+\eps^4(2\pi)^6\,\mathbb{B}^*_4\mathbb{B}^*_2\cdot\eta^{\otimes 6},
\end{align*}
which, together with Parseval's identity, provides the claimed result. 
\end{proof}

\section{Outlook}

One of the main applications of Bloch waves in the context of homogenization is the long time behavior of waves in $\eps$-periodic media. 
Actually, while for fixed times $t\in[0,T]$ the constant coefficient effective wave equation 
\begin{equation}
\label{eq:effectivewaves}
\partial^2_t u-\nabla\cdot(A^{\text{eff}}\nabla u)=f
\end{equation}
provides a good approximation of the original oscillatory model, for long times 
$t\in[0,T/\eps^2]$ dispersive effects set in, which are not captured by the standard effective equation \eqref{eq:effectivewaves}. 
In the articles \cite{AP,AP2,AP3,ALR,Lam} long time effective disersive models are derived by means of two-scale expansion and in \cite{DLS,DLS2,santosa} by means  
of Bloch wave analysis. We believe that the approach of this paper can also be applied to optimal control problems involving oscillatory wave equations on large time intervals. 

Another generalization regards the control problem of \cite{KP1}, where the oscillatory matrices in the PDE-constraint and in the energy may differ, 
cf. the discussion in the introduction. In a future project we hope to extend the Bloch approach to this general framework.

\end{document}